\documentclass[reqno, 11pt]{amsart}
\usepackage{amssymb,amsmath,amsfonts}
\usepackage[margin=1in]{geometry}
\usepackage{dsfont}
\usepackage{color}
\usepackage{graphicx}
\usepackage{latexsym}
\usepackage{amsmath}
\usepackage{amssymb}
\usepackage{graphics}
\usepackage[dvips]{epsfig}
\usepackage{mathrsfs}
\usepackage{mathtools}
\usepackage{leqno}
\usepackage{stmaryrd,cite}

\usepackage{cases}
\usepackage{enumerate}
\usepackage[usenames,dvipsnames,svgnames]{xcolor}
\definecolor{refkey}{rgb}{1,0,0.5}
\definecolor{labelkey}{rgb}{0,0.4,1}
\usepackage{todonotes}
\usepackage{marginnote}

\presetkeys{todonotes}{fancyline, color=green!40}{}

\makeatletter
\renewcommand{\@todonotes@drawMarginNoteWithLine}{%
	\begin{tikzpicture}[remember picture, overlay, baseline=-0.75ex]%
	\node [coordinate] (inText) {};%
	\end{tikzpicture}%
	\marginnote[{
		\@todonotes@drawMarginNote%
		\@todonotes@drawLineToLeftMargin%
	}]{
		\@todonotes@drawMarginNote%
		\@todonotes@drawLineToRightMargin%
	}%
}
\makeatother
\usepackage{listings}
\usepackage{ifpdf}
\ifpdf
\usepackage[CJKbookmarks=true,
         hyperindex=true,
         pdfstartview=FitH,
         bookmarksnumbered=true,
         bookmarksopen=true,
         colorlinks=true,
         citecolor=blue,
         linkcolor=blue,
         urlcolor=blue,
         pdfborder=001,
         pdfauthor={},
         pdftitle={},
         pdfkeywords={},
         ]{hyperref}
\else
\usepackage[
         hypertex,
         hyperindex,
         linkcolor=blue,
         unicode,
         citecolor=blue%
         ]{hyperref}
\fi
\usepackage{cleveref}

\allowdisplaybreaks

\numberwithin{equation}{section}
\newtheorem{thm}{Theorem}[section]

\newtheorem{lem}[thm]{Lemma}
\newtheorem{prop}[thm]{Proposition}
\newtheorem{rmk}[thm]{Remark}

\newcommand{\be}{\begin{equation}}
\newcommand{\ee}{\end{equation}}

\newcommand{\bee}{\begin{equation*}}
\newcommand{\eee}{\end{equation*}}

\newcommand{\bse}{\begin{subequations}}
\newcommand{\ese}{\end{subequations}}

\newcommand{\bs}{\begin{split}}
\newcommand{\es}{\end{split}}

\newcommand{\ef}{\eqref}

\begin{document}

\author{Hairong Liu$^{1}$}\thanks{$^{1}$School of Science, Nanjing Forestry University, Nanjing 210037, P.R.China.
E-mail : hrliu@njfu.edu.cn}

\author{Tao Luo$^{2}$}\thanks{$^{2}$Department of Mathematics, City University of Hong Kong, 83 Tat Chee Avenue, Kowloon Tong, Hong Kong.
E-mail: taoluo@cityu.edu.hk}

\author{Hua Zhong$^{3}$}\thanks{$^{3}$School of Mathematics, Southwest Jiaotong University, Chengdu 611756, P.R. China.
E-mail: huazhong@swjtu.edu.cn}

\title[] {Global Solutions to an initial boundary problem  for the compressible 3-D MHD equations with Navier-slip and perfectly conducting boundary
conditions in exterior domains}

\begin{abstract}

 An initial boundary value problem for compressible  Magnetohydrodynamics (MHD) is considered on an exterior domain (with the first Betti number vanishes) in $\mathbb{R}^3$ in this paper. The global existence  of smooth solutions near a given constant  state for compressible MHD with the boundary conditions of Navier-slip for the velocity filed  and perfect conduction  for the magnetic field is established.  Moreover the explicit decay rate is given. In particular, the results obtained in this paper also imply the global existence of classical solutions for the full compressible Navier-Stokes equations with Navier-slip boundary conditions on exterior domains in three dimensions, which is not available in literature, to the best of knowledge of the authors'.

\end{abstract}

\maketitle

\section{Introduction and Main Theorems}

Magnetohydrodynamics (MHD) mainly investigates the dynamics of compressible quasineutrally ionized fluids under the influence of electromagnetic fields. It is well-known that the applications of Magnetohydrodynamics cover a very wide range of physical objects, from liquid metals to cosmic plasmas. In this paper,
the 3-dimensional  compressible full magnetohydrodynamic equations will be considered, which take the following form (\cite{Laudau,HW2008}):
\begin{eqnarray}{\label{prob2}}
\left\{
\begin{array}{llll}
\rho_{t}+\mbox{div}(\rho u)=0,\\[2mm]
\rho\big(u_{t}+u\cdot\nabla u\big)-\mu\Delta u-(\mu+\lambda)\nabla\mbox{div}u+\nabla p=\mbox{curl}H\times H,\\[2mm]
\partial_{t}(\rho e)+\mbox{div}(\rho e u)+p\mbox{div}u=\Psi:\nabla u+\mbox{div}(\kappa\nabla\mathcal{T})+\eta|\mbox{curl}H|^2,\\[2mm]
H_{t}- \eta \Delta H=\mbox{curl}(u\times H),\quad \quad \mbox{div}H=0,
\end{array}
\right.
\end{eqnarray}
where $\rho$, $u=(u_1,u_2,u_3)$, $H=(H_1,H_2,H_3)$, $e$, $\mathcal{T}$ denote the density, the velocity, the magnetic field, the internal energy  and temperature, respectively.
The viscosity coefficients $\mu$ and $\lambda$ of the fluid should satisfy
$$2\mu+3\lambda>0,\quad \mu>0,$$
due to physical realities. $\kappa>0$ is the heat conductivity, and $\eta>0$ is the magnetic diffusivity which describe a magnetic diffusion coefficient of the magnetic field. $\Psi$ is the viscous stress tensor, given by
\begin{equation*}
\Psi=\mu(\nabla u+\nabla u^{T})+\lambda \mbox{div}uI.
\end{equation*}
and  $\Psi:\nabla u$ denotes the scalar product of two matrices. A calculation gives  that
\begin{equation*}
\Psi:\nabla u=\lambda(\mbox{div}u)^2+\sum_{i,j=1}^{3}\frac{\mu}{2}\left(\frac{\partial u^{i}}{\partial x_{j}}+\frac{\partial u^{j}}{\partial x_{i}}\right)^2=\lambda(\mbox{div}u)^2+2\mu |S(u)|^2.
\end{equation*}
For simplicity, we study the case of  polytropic ideal gas type with the equations of state:
\begin{equation*}
	e=c_{v}\mathcal{T}, \quad p=R\rho\mathcal{T},
\end{equation*}
where $c_v$ and $R$ are positive constants. In this case, we may write
 $(\ref{prob2})_3$ as
\begin{equation*}
c_{v}\rho(\mathcal{T}_{t}+ u\cdot\nabla\mathcal{T})+p\mbox{div}u=\kappa\Delta\mathcal{T}+\Psi:\nabla u+\eta|\mbox{curl}H|^2.
\end{equation*}

Let $U$ be a simply connected bounded smooth domain in $\mathbb{R}^{3}$, and $\Omega\equiv\mathbb{R}^{3}\backslash \bar{U}$ be the exterior domain. In this paper, we study the initial-boundary value problem of (\ref{prob2}) in $\Omega$ with the initial condition
 \begin{equation}\label{initial}
 (\rho, u, \mathcal{T}, H)|_{t=0}=(\rho_0, u_0,  \mathcal{T}_0, H_0),
 \end{equation}
 and the boundary conditions:
 \begin{equation}\label{boundary2}
 	u\cdot n=0,\quad \mbox{curl}u\times n=0,\quad \mbox{on}\quad x\in\partial\Omega,
 \end{equation}
 \begin{equation}\label{boundary3}
\quad H\cdot n=0,\quad \mbox{curl}H\times n=0,\quad \frac{\partial\mathcal{T}}{\partial{n}}=0, \quad \mbox{on}\quad x\in\partial\Omega.
\end{equation}
Moreover, besides the  the compatibility condition of the initial data with the boundary condition, we also assume that
the initial data satisfy the compatibility conditions,
\begin{equation*}
	(\rho_0, u_0,  \mathcal{T}_0, H_0)(x)\rightarrow (1,0,1,0),\quad\quad\text{as}\quad x\rightarrow\infty.
\end{equation*}
The boundary condition \ef{boundary2} means the Navier-slip boundary conditions for the velocity field, \ef{boundary2} is  the perfectly conducting boundary condition for  the magnetic field, and Neumann boundary condition for  temperature.

The Navier-slip boundary conditions for the velocity field were first introduced by Navier in \cite{Navierslip} and have been used in  many applications, which  are usually used in the large eddy simulations of turbulent flows,  to compute the large eddies of a turbulent flow accurately by neglecting small flow structure. For this, the slip boundary conditions are more suitable than the Dirichlet boundary conditions \cite{GL2000}.

The system of  compressible  MHD equations have been studied extensively by physicists and mathematicians because of its physical importance, rich and complex phenomena and mathematical challenges. We mainly review the results on the global existence of solutions related to the main theme of this paper, for which one may refer to \cite{chenwang2002,wang2003,Kawashima1982} for one-dimensional case and \cite{Ducomet2006, Kawashima1982, HW2008, LXZ2013,PuG2013,HW2011,LY2011,ZZ2008} for higher dimensions for example. The results for the global existence and large time behavior of solutions in 3-D for  the Cauchy problem , can be found, for instance, the  classical solutions in \cite{Kawashima1982,LXZ2013,PuG2013, LY2011,ZZ2008},  the variational weak solutions  in \cite{HW2008,HW2011,Ducomet2006}.

In the presence of physical boundaries,  there have been extensive studies on the global existence and large time behavior of solutions to the initial boundary value problems for compressible MHD equations in 3-D  for the homogenous Dirichlet condition for both velocity filed $u$ and the magnetic field $H$, i.e., $u=0$  and $H=0$ on the boundary. It is more physically relevant to consider the case that the magnetic field is non-zero on the boundary. In this direction, the global existence of classical solutions for 3-D compressible isentropic MHD was proved in \cite{CHS2021-1} very recently in a bounded domain (see also \cite{DJJ2013} for the related results in 2-D and \cite{CL2021} for the related results for compressible isentropic Navier-Stokes equations in 3-D).  For the compressible Navier-Stokes equations in 3-D with Navier-slip boundary conditions,  there are only few results on the global solutions for the initial boundary value problem for the Navier-slip boundary conditions in 3-D, though results are available for some small physical parameter limits for local time, for instance, see  the related zero viscosity limit \cite{XX2007, XXW2009, 2012CMP,Masmoudi2012,  Wang2016} and \cite{COR2015,GLX2019} for low Mach number limit.

Compared with the problem for 3-D compressible isentropic MHD equations studied  in \cite{CHS2021-1}  in a bounded domain, the difficulty in proving the global existence of classical solution for the initial boundary value problem of
(\ref{prob2})-(\ref{boundary3}) in an  exterior domain in $\mathbb{R}^3$ is due to the unboundedness of the exterior domain, for which  the Poincare type inequality to use the $L^2$-norm of the derivatives to control the $L^2$-norm of the solution itself is not available. For example, in a bounded domain,  the dissipation estimates (the $L^2(\Omega\times [0, \infty)$ space-time estimates)  of the velocity and magnetic fields can be obtained via the corresponding estimates of the derivatives which follow from the dissipation of the viscosity and magnetic diffusion in  the basic energy estimates. This is an important decay mechanism for the problem on a bounded domain for which various exponential decay  estimates are obtained in \cite{CHS2021-1}. However, for the problem on an exterior domain studied in the present  paper, only algebraic decay can be expected  in general. Another difference from the case studied in \cite{CHS2021-1} is that we can handle the case of the full compressible MHD system with variable entropy,  while only isentropic case is studied in  \cite{CHS2021-1} for which possibly large oscillations and vacuum states (with small
energy) are allowed. It should be noted that,  for non-isentropic flows,  even in the case without magnetic fields ($H=0$),  there have been no global-in-time theory available for the problems with physical boundaries under  the Navier-slip boundary conditions. 

For  the boundary conditions of Navier-slip for the velocity field and the perfect conduction of magnetic filed, it is more suitable to use the $L^2$-norms of $div$ and $curl $ to control that of the derivatives of the velocity and magnetic fields.  This control usually depends on the topology of the domain, which relates to the first Betti number. A key observation is that first Betti number of  the exterior domain $\Omega=\mathbb{R}^3\backslash \bar U$ vanishes when $U$ is a bounded simply connected open set.  In this case,   one can apply the refined div-curl estimate (see Proposition \ref{prop-imp}) to  obtain  the dissipation estimate of $\|\nabla u\|$ with the help of the   dissipation estimates for $\|\mbox{curl}u\|$ and $\|\mbox{div}u\|$. This observation is important to obtain the various dissipation estimates.

The basic approach to obtain the global existence and decay estimates of classical  solutions of the initial boundary value problem is based on the energy estimates, for which the boundary conditions (\ref{boundary2})-(\ref{boundary3}) make the problem rather interesting and challenging. In fact, unlike the homogeneous Dirichlet boundary conditions $u=H=0$ on the boundary, various estimates on $div u$, $curl u$ and $curl H$ are crucial for the problem investigated in this paper.

The main results of this paper are stated as follows.
\begin{thm}\label{thm-1}  Suppose that the initial data satisfy the compatibility condition with the boundary conditions.  There exists a constant $\delta_1>0$ such that if
\begin{equation*}
\|\rho_0-1\|_2+\|(u_0, \mathcal{T}_0-1,H_0)\|_{3}\leq \delta_1,
\end{equation*}
then the initial boundary value problem (\ref{prob2})-(\ref{boundary3}) admits a  unique strong solution
$(\rho,u,\mathcal{T},H)$ globally in time  satisfying
\begin{equation*}
\rho-1\in C^{0}(0,\infty,H^{2}(\Omega))\cap C^{1}(0,\infty,H^{1}(\Omega)),
\end{equation*}
\begin{equation*}
(u, \mathcal{T}-1, H)\in C^{0}(0,\infty,H^{3}(\Omega))\cap C^{1}(0,\infty,H^{1}(\Omega)),
\end{equation*}
and
\begin{equation*}
\begin{split}
&\Big(\|(u,\mathcal{T}-1,H)\|^2_{3}+\|\rho-1\|^2_{2}+\|(\rho_t,u_t, H_t,\mathcal{T}_t)\|^2_{1}\Big)(t)\\[2mm]
&+C\int_{0}^t\Big(
\|(\nabla u,\nabla\mathcal{T},\nabla H)\|^2_{2}+\|\nabla \rho\|^2_{1}+\|(u_{t},\mathcal{T}_{t},H_{t})\|^2_{2}+\|\rho_{t}\|^2_{1}\Big)ds \\[2mm]
&\leq \Big(\|(u,\mathcal{T}-1,H)\|^2_{3}+\|\rho-1\|^2_{2}+\|(\rho_t,u_t, H_t,\mathcal{T}_t)\|^2_{1}\Big) (0).
\end{split}
\end{equation*}
	where $C>0$ is a positive constant independent of $t$.
\end{thm}
\begin{rmk}
 The compatibility condition of the initial data with the boundary conditions is standard to ensure the regularity of the solutions.
The data
\begin{equation*}
\|(\rho_t,u_t, H_t,\mathcal{T}_t)\|^2_{1}(0)
\end{equation*}
is determined  by the the initial data $(\rho_0,u_0,\mathcal{T}_0,H_0)$ through the equations (\ref{prob2}).
\end{rmk}

For the initial data  close to the constant state $(1, 0,1,0)$ in higher-order Sobolev norm, we can improve the regularity of solutions  in Theorem \ref{thm-1}. Precisely,
\begin{thm}\label{thm-2} { Suppose that the initial data satisfy the compatibility condition with the boundary conditions. }There exists a constant $\delta>0$ such that if
\begin{equation*}
\|\rho_0-1\|_3+\|(u_0, \mathcal{T}_0-1,H_0)\|_{4}\leq \delta,
\end{equation*}
then the initial boundary value problem (\ref{prob2})- (\ref{boundary3}) admits a unique  smooth solution
$(\rho,u,\mathcal{T},H)$ globally in time  satisfying
\begin{equation*}
\rho-1\in C^{0}(0,\infty, H^{3}(\Omega))\cap C^{1}(0,\infty, H^{2}(\Omega)),
\end{equation*}
\begin{equation*}
(u, \mathcal{T}-1, H)\in C^{0}(0,\infty,H^{4}(\Omega))\cap C^{1}(0,\infty,H^{2}(\Omega)),
\end{equation*}
and
\begin{equation*}
\begin{split}
&\Big(\|\rho-1\|^2_3+\|(u,\mathcal{T}-1, H)\|^2_{4}+\|(\rho_t,u_t,\mathcal{T}_t,H_t)\|^2_{2}+\|(\rho_{tt},u_{tt},\mathcal{T}_{tt},H_{tt})\|^2\Big)(t)\\[2mm]
&+C\int_{0}^t\Big(
\|(\nabla u,\nabla\mathcal{T},\nabla H)\|^2_{3}+\|\nabla \rho\|^2_{2}+\|(\rho_t,u_{t},H_{t})\|^2_{2}+\|\mathcal{T}_{t}\|^2_{3}
+\|(u_{tt},\mathcal{T}_{tt},H_{tt})\|^2_{1}+\|\rho_{tt}\|^2\Big)ds \\[2mm]
&\leq \|\rho-1\|^2_3+\|(u,\mathcal{T}-1, H)\|^2_{4}+\|(\rho_t,u_t,\mathcal{T}_t,H_t)\|^2_{2}+\|(\rho_{tt},u_{tt},\mathcal{T}_{tt},H_{tt})\|^2\Big)(0),
\end{split}
\end{equation*}
	where $C>0$ is a positive constant independent of $t$.
\end{thm}

In addition, we shall show that the solution in Theorem \ref{thm-2} approaches the stationary state as $t\rightarrow \infty$, and give the explicit decay rate. More precisely,
\begin{thm}\label{thm-3}
Let $(\rho,u,\mathcal{T}, H)$ be the global solution in Theorem  \ref{thm-2}, and $(1,u_s,1,H_s)$ with $u_s=H_s=(0,0,0)$  be a stationary state. Then it holds:
\begin{equation*}
\begin{split}
\|(\rho_t, u_t, \mathcal{T}_t, H_t)(t)\|=O(t^{-1/2})\quad \mbox{as}\quad t\rightarrow \infty\\[2mm]
\|(\nabla \rho,\nabla u, \nabla \mathcal{T},\nabla H)(t)\|=O(t^{-1/4}), \quad \mbox{as}\quad t\rightarrow \infty\\[2mm]
\|(\nabla^2u, \nabla^2 \theta,\nabla^2H)(t)\|=O(t^{-1/4}), \quad \mbox{as}\quad t\rightarrow \infty,\\[2mm]
\|(u,\mathcal{T}-1, H)(t)\|_{C^{0}(\Omega)}=O(t^{-1/4})\quad \mbox{as}\quad t\rightarrow \infty\\[2mm]
\|(\rho-1)(t)\|_{C^{0}(\Omega)}=O(t^{-1/8})\quad \mbox{as}\quad t\rightarrow \infty.
\end{split}
\end{equation*}
\end{thm}
\begin{rmk}
Note that for the IBVP of the Navier-Stokes equations with Dirichlet condition $u|_{\partial\Omega}=0$ on exterior domain, the decay estimates of solutions to stationary state could be found in \cite{1992MZ,Deckelnick1993,Kobayashi1999,Matsumura1983,Jiang1996}. By contrast,  the big difference in this paper is that we consider the Navier-slip boundary condition for $u$, instead of the homogeneous Dirichlet condition.  Inspired by \cite{1992MZ}, we would like to form the inequalities as in lemma \ref{lem-deacy}. Different from \cite{1992MZ}, Lemma \ref{lem4-3} is used to obtain the estimate for $\|\mbox{curlcurlu}\|$, which is important for us to derive the decay estimates of $\|(\nabla^2u,\nabla\rho)\|$ (Lemma \ref{lem4-4}). Moreover,  some classical elliptic estimates will be applied for $H$ to find the corresponding decay rate.
\end{rmk}

\begin{rmk}
The algebraic decay for Cauchy problem of MHD equations was shown  in \cite{chentan2010,LY2011,ZZ2008} by careful analysis of the  linearized equation by using the Fourier transformation. However, due to the presence of the physical boundary, the approach of obtaining the decay for the Cauchy problem used in \cite{chentan2010,LY2011,ZZ2008} cannot be applied to the problem studied in this paper.   Moreover, the conditions required for initial data are lesser in the present paper.
\end{rmk}

This paper is organized as follows.  Section 2 contains some notations and basic lemmas for later use. In Section 3 we prove the global existence of strong solutions.
The improved  regularity of the strong solutions in Theorem \ref{thm-1} will be proved in Sec. 4.
Finally, Section 5 is devoted to obtaining the decay rates.
Through out of this paper, we introduce some notations for later use.

The $L^p-$norm on $x\in\Omega$ for $\phi(x,t)$ is given by
\begin{equation*}
	\begin{split}
		\|\phi\|_{L^{p}}\equiv\|\phi\|_{L^{p}(\Omega)}=\left(\int_\Omega|\phi|^p(x,t)dx\right)^{1/p},\,\,\,\|\phi\|\equiv\|\phi\|_{L^{2}},
	\end{split}
\end{equation*}
$1\leq p<+\infty$ and $H^m$ is used to denote the standard Sobolev space with the following norm
\begin{equation*}
	\|\phi\|_{m}\equiv\|\phi\|_{H^{m}(\Omega)}=\left(\sum_{l=0}^m\|\nabla^l\phi\|^2\right)^{1/2}.
\end{equation*}
$H^{m,p}$ is applied to express the standard Sobolev space with the following norm
\begin{equation*}
	\|\phi\|_{H^{m,p}}\equiv\left(\sum_{l=0}^m\|\nabla^l\phi\|^p_p\right)^{1/p}.
\end{equation*}
Besides, $C$ will be used as a generic constant independent of time $t$.

\section{Notations and basic lemmas}

In this section, we list some lemmas which will be frequently used throughout this paper.

Firstly, we recall some inequalities of Sobolev type.
\begin{lem}\label{sobolev inequ}
Let $\Omega\subset\mathbb{R}^{3}$ be a domain with smooth boundary. Then
\begin{equation*}
\begin{array}{lll}
(i)\ \|f\|_{C^{0}(\bar{\Omega})}\leq C\|f\|_{H^{m,p}}\quad \mbox{for}\quad f\in H^{m,p}(\Omega),\quad mp>3>(m-1)p.\\[2mm]
(ii)\ \|f\|_{L^{p}}\leq C\|f\|_{1}\quad \mbox{for}\quad f\in H^{1}(\Omega), \quad 2\leq p\leq 6.\\[2mm]
(iii)\ \|f\|_{L^{6}}\leq C\|\nabla f\|\quad \mbox{for}\quad f\in H^{1}(\Omega).
\end{array}
\end{equation*}
\end{lem}

The following lemma  allows one to control the $H^{m}$-norm of a vector valued function $v$ by its
$H^{m-1}$-norm of $\mbox{curl}v$ and $\mbox{div} v$ (see \cite{XX2007}).
\begin{lem}\label{lem-div-curl}
Let $\Omega$ be a domain in $\mathbb{R}^{N}$ with smooth boundary $\partial\Omega$ and outward normal $n$. Then there exists a constant $C>0$,
such that
\begin{equation}\label{A3-1}
\|v\|_{H^{s}(\Omega)}\leq C\big(\|\mbox{div}v\|_{H^{s-1}(\Omega)}+\|\mbox{curl}v\|_{H^{s-1}(\Omega)}
+ |v\cdot n|_{H^{s-1/2}(\partial\Omega)}+\|v\|\big),
\end{equation}
and
\begin{equation}
\|v\|_{H^{s}(\Omega)}\leq C\big(\|\mbox{div}v\|_{H^{s-1}(\Omega)}+\|\mbox{curl}v\|_{H^{s-1}(\Omega)}
+ |v\times n|_{H^{s-1/2}(\partial\Omega)}+\|v\|\big),
\end{equation}
for any $v\in H^{s}(\Omega)$, $s\geq1$.
\end{lem}

In addition, since the first Betti number of the exterior domain $\Omega$ in our paper vanishes,   we can apply Theorem 3.2 in \cite{wahl} to obtain the following  refined estimate  of $\|\nabla v\|$ which is crucial in the case of  exterior domain. We also note that the topological property of $\Omega$ is necessary  for the following  Proposition.

\begin{prop} (\cite{wahl})\label{prop-imp} Let $\nabla v\in L^{2}(\Omega)$, $v\in L^{6}(\Omega)$ and $v\cdot n|_{\partial\Omega}=0$. The estimate  (C independent of $v$)
\begin{equation}\label{important}
\|\nabla v\|\leq C_{\Omega}(\|\mbox{div}v\|+\|\mbox{curl}v\|)
\end{equation}
is true for all $v$ as above if and only if $\Omega$ has a first Betti number of zero.
\end{prop}

Next, we list some elliptic estimates for elliptic equations with Navier-slip  or Neumann boundary conditions in the smooth exterior domain $\Omega$.
\begin{lem} (\cite{2011SWZ})\label{lem-elliptic-1}
	Let $s$ be an  integer $\geq0$. Suppose that $u$  satisfy the Lam\'{e} equation:
	\begin{equation*}
		\left\{
		\begin{array}{llll}
			-\mu\Delta u-(\mu+\lambda)\nabla\mbox{div}v=f,\quad \mbox{in}\quad \Omega,\\[2mm]
			u\cdot n=0,\quad \mbox{curl} u\times n=0\quad \mbox{on}\quad \partial\Omega
		\end{array}
		\right.
	\end{equation*}
	Then, it holds
	\begin{equation*}
		\|\nabla^2u\|_{H^{s}(\Omega)}\leq C\big(\|f\|_{H^{s}(\Omega)}+\|\nabla u\|_{L^{2}(\Omega)}\big).
	\end{equation*}
In particular,  the same conclusion also holds for Laplace operator, that is, for
\begin{equation*}
\quad v\cdot n=0,\quad \mbox{curl} v\times n=0\quad \mbox{on}\quad \partial\Omega,
\end{equation*} it holds
\begin{equation*}
\|\nabla^2v\|_{H^{s}(\Omega)}\leq C\big(\|\Delta v\|_{H^{s}(\Omega)}+\|\nabla v\|_{L^{2}(\Omega)}\big).
\end{equation*}
\end{lem}

\begin{lem}(\cite{2019JMAA}) \label{lem-elliptic-2}
Let $s$ be an  integer $\geq0$. Suppose that $(v,p)$ is the solution of the   Stokes problem:
\begin{equation*}
\left\{
\begin{array}{llll}
\mbox{div}v=g,\quad \mbox{in}\quad \Omega,\\[2mm]
-\Delta v+\nabla p=f,\quad \mbox{in}\quad \Omega,\\[2mm]
v\cdot n=0,\quad \mbox{curl} v\times n=0\quad \mbox{on}\quad \partial\Omega
\end{array}
\right.
\end{equation*}
Then, it holds
\begin{equation*}
\|\nabla^2v\|_{H^{s}(\Omega)}+\|\nabla p\|_{H^{s}(\Omega)}\leq C\big(\|f\|_{H^{s}(\Omega)}+\|g\|_{H^{s+1}(\Omega)}+\|\nabla v\|_{L^{2}(\Omega)}\big).
\end{equation*}
\end{lem}

\begin{lem}(\cite{1997JMPA})\label{lem-elliptic-3}
For any function $\theta\in H^{s}(\Omega)$ with $\frac{\partial \theta}{\partial n}|_{\partial\Omega}=0$, then it holds
\begin{equation*}
\|\nabla^2\theta\|_{H^{s}(\Omega)}\leq C\big(\|\Delta\theta\|_{H^{s}(\Omega)}+\|\nabla \theta\|_{L^{2}(\Omega)}\big).
\end{equation*}
\end{lem}

Finally, we recall the following differential inequalities which is the main argument in proving Theorem \ref{thm-2}. The reader is referred to \cite{M1984} for a proof.
\begin{lem}\label{lem-deacy}
(i) Let $g\in C^{1}([t_0,\infty))$ such that $g\geq0$, $E=\int_{t_0}^{\infty}g(t)dt<\infty$ and
\begin{equation*}
g'(t)\leq a(t)g(t) \quad \mbox{for all}\quad t\geq t_0
\end{equation*}
where $a\geq0$, $M=\int_{t_0}^{\infty}a(t)dt<\infty$. Then
\begin{equation*}
g(t)\leq \Big((t_0g(t_0)+1)exp(E+M)-1\Big)t^{-1}\quad \forall t\geq t_0.
\end{equation*}
(ii) Let $g\in C^{1}([t_0,\infty))$ such that $g\geq0$ and
\begin{equation*}
g'(t)+c_0g(t)\leq c_1t^{-\alpha} \quad \mbox{for all}\quad t\geq t_0
\end{equation*}
with positive constants $c_0,c_1$ and $\alpha$. Then there exists a $t_1=t_1(c_0,\alpha)\geq t_0$ such that
\begin{equation*}
g(t)\leq Ct^{-\alpha},\quad \mbox{for} \quad t\geq t_{1}.
\end{equation*}
\end{lem}

\section{Global existence  of strong solutions}

Let
\begin{equation*}
\rho=1+ q, \mathcal{T}=1+\theta,
\end{equation*}
 from the initial boundary value problem (\ref{prob2}), then $(q, u, \theta,H)$   satisfy the following system

\begin{equation}\label{prob-linear}
\left\{
\begin{array}{llll}
q_{t}+\mbox{div}u=-\mbox{div}(q u),\\[2mm]
\rho\big(u_{t}+u\cdot\nabla u\big)-\mu\Delta u-(\mu+\lambda)\nabla\mbox{div}u+R\nabla q+R\nabla\theta=-R\nabla(q\theta)+\mbox{curl}H\times H,\\[2mm]
c_{v}\rho(\theta_{t}+ u\cdot\nabla\theta)-\kappa\Delta\theta+R\mbox{div}u=\lambda(\mbox{div}u)^2+2\mu |S(u)|^2-R(\rho\theta+q)\mbox{div}u+\eta|\mbox{curl}H|^2,\\[2mm]
H_{t}- \eta \Delta H=\mbox{curl}(u\times H),\quad \quad \mbox{div}H=0,\\[2mm]
\end{array}
\right.
\end{equation}
with the initial data
\begin{equation}\label{in}
	(q, u, \theta, H)|_{t=0}=(q_0, u_0,  \theta_0, H_0)	
	\end{equation}
and the boundary conditions
\begin{equation}\label{boundary-linear}
u\cdot n=0,\quad \mbox{curl}u\times n=0, \quad H\cdot n=0,\quad \mbox{curl}H\times n=0,\quad \frac{\partial\theta}{\partial{n}}=0, \quad \mbox{on}\quad x\in\partial\Omega.
\end{equation}

Theorem \ref{thm-1} will be proved in this section. The local-in-time well-posedness in the smooth norm is quite standard,  therefore to prove Theorem \ref{thm-1}, it suffices to prove the following {\it a priori} estimates.
In order to describe clearly, we set  $\mathcal{E}_1(t)$ and $\mathcal{D}_1(t)$ as follows:
\begin{equation}\label{E3}
	\mathcal{E}_1(t)=\|(u,\theta,H)\|_{3}+\|q\|_{2}+\|(q_t,u_t, H_t,\theta_t)\|_{1}
\end{equation}
and
\begin{equation}\label{D3}
	\mathcal{D}_1(t)=\|(\nabla u,\nabla\theta,\nabla H)\|_{2}+\|\nabla q\|_{1}+\|(u_{t},\theta_{t},H_{t})\|_{2}+\|q_{t}\|_{1}+
\|(q_{tt},u_{tt},\theta_{tt},H_{tt})\|.
\end{equation}

\begin{prop}\label{prop} ({\it {a priori estimates}})
	Let $(q,u,\theta, H)$ be a solution to the initial boundary value problem (\ref{prob-linear}), (\ref{in}) and (\ref{boundary-linear}) in  $t\in[0,T]$. Then there exists  positive constants $C$ and $\delta_1$ which are independent of $t$,  such that if
	\begin{equation*}
	\sup_{0\leq t\leq T}\mathcal{E}_1(t)\leq \delta_1,
	\end{equation*}
	then there holds, for any $t\in[0,T]$,
	\begin{equation*}
	\mathcal{E}^2_1(t)+C\int_{0}^{t}\mathcal{D}^2_1(s)ds\leq C\mathcal{E}^2_1(0).
	\end{equation*}
\end{prop}

Now, we shall divide the
proof of Proposition \ref{prop}  into several lemmas.

\subsection{Estimates for lower-order derivatives}\label{subsection3-1}
 To begin with, we have the following basic energy estimate.
\begin{lem}\label{lem-1}
\begin{equation*}
	\|(q,u,\theta,H)\|^2+c_0\int^t_0\|(\nabla u,\nabla\theta,\nabla H)\|^2ds\leq \|(q,u,\theta,H)\|^2(0)+ C\delta_1\int^t_0\mathcal{D}_1^2(s)ds,
\end{equation*}	
where $c_0>0$, $C>0$ are positive constants independent of $t$.
\end{lem}
\begin{proof}
Computing the following integral
\begin{equation*}
\int_{\Omega}\left\{(\ref{prob-linear})_1Rq+(\ref{prob-linear})_2\cdot u+(\ref{prob-linear})_3\theta+(\ref{prob-linear})_4\cdot H\right\} dx,
\end{equation*}
and noting that $\Delta u=-\mbox{curl}^2u+\nabla \mbox{div}u$,
after integrating by parts, one has
\begin{equation}\label{1}
\begin{split}
&\frac{1}{2}\frac{d}{dt}\int_{\Omega} \Big(Rq^2+\rho|u|^2+c_{v}\rho\theta^2+|H|^2\Big)dx
+\mu\|\mbox{curl}u\|^2+(2\mu+\lambda)\|\mbox{div}u\|^2+\kappa\|\nabla\theta\|^2+\eta\|\mbox{curl}H\|^2\\[2mm]
&=\int_{\Omega}\left(\Big(-\frac{1}{2}Rq^2-R\rho\theta^2\Big)\mbox{div}u+\Big(\lambda(\mbox{div}u)^2+2\mu |S(u)|^2+\eta|\mbox{curl}H|^2\Big)\theta\right)dx\\[2mm]
&\leq C\mathcal{E}(t)\mathcal{D}_1^2(t).
\end{split}
\end{equation}
Applying  Proposition \ref{prop-imp} to $u$ and $H$, it holds that
\begin{equation}\label{u}
\|\nabla u\|^2\leq C(\|\mbox{curl}u\|^2+\|\mbox{div}u\|^2),
\end{equation}
and
\begin{equation}\label{H}
\|\nabla H\|^2\leq C\|\mbox{curl}H\|^2.
\end{equation}
Therefore, integrating (\ref{1}) over $[0,t]$,  and noting that $2\mu+\lambda>0$, we prove the basic energy estimate.
\end{proof}

Next, we prove estimates of the first-order  derivatives $(\nabla u,\nabla \theta,\nabla H)$.
\begin{lem}\label{lem-ut}
	\begin{equation*}
		\|(\nabla u,\nabla\theta,\nabla H)\|^2+c\int^t_0\|(u_t,q_t,\theta_t,H_t)\|^2ds\leq C\mathcal{E}^2_1(0)+C\delta_1\int^t_0\mathcal{D}_1^2(s)ds,
	\end{equation*}	
	where $c>0$, $C>0$ are positive constants independent of $t$.	

\end{lem}
\begin{proof}
Computing the following integral
	\begin{equation*}
\int_{\Omega}\left\{(\ref{prob-linear})_2\cdot u_t+(\ref{prob-linear})_1\cdot  q_t+(\ref{prob-linear})_3\cdot \theta_t+(\ref{prob-linear})_4\cdot H_t\right\}dx
\end{equation*}
integrating by parts, one has
\begin{equation*}
\begin{split}	
&\frac{1}{2}\frac{d}{dt}\Big\{\mu\|\mbox{curl}u\|^2+(2\mu+\lambda)\|\mbox{div}u\|^2+\kappa\|\nabla\theta\|^2+\eta\|\mbox{curl}H\|^2\Big\}
-\frac{d}{dt}\int_{\Omega}R(q+\theta)\mbox{div} u dx\\[2mm]
&+\|\sqrt{\rho}u_t\|^2+\|q_t\|^2+c_{v}\|\sqrt{\rho}\theta_t\|+\|H_t\|^2\\[2mm]
&=\int_{\Omega}\Big((R+1) q_{t}+2R\theta_{t}\Big)\mbox{div}u dx
+\int_{\Omega}\mbox{div}(qu) q_{t}dx
+\int_{\Omega}\Big(\rho u\cdot\nabla u-R\nabla (q\theta)+\mbox{curl}H\times H\Big)\cdot u_{t}dx\\[2mm]
&+\int_{\Omega}\Big(-c_{v}\rho u\cdot\nabla \theta+\lambda(\mbox{div}u)^2+2\mu|Su|^2-R(\rho\theta+q)\mbox{div}u+\eta|\mbox{curl}H|^2\Big)\theta_{t}dx
+\int_{\Omega}\mbox{curl}(u\times H)\cdot H_{t}dx\\[2mm]
&\leq \frac{1}{2}\Big(\|\sqrt{\rho}u_t\|^2+\|q_t\|^2+c_{v}\|\sqrt{\rho}\theta_t\|+\|H_t\|^2\Big)+C\|\mbox{div}u\|^2
+\mathcal{E}_1(t) \mathcal{D}^2_1(t),
\end{split}
\end{equation*}
which gives
\begin{equation}\label{above-1}
\begin{split}	
&\frac{1}{2}\frac{d}{dt}\Big\{\mu\|\mbox{curl}u\|^2+(2\mu+\lambda)\|\mbox{div}u\|^2+\kappa\|\nabla\theta\|^2+\eta\|\mbox{curl}H\|^2\Big\}
-\frac{d}{dt}\int_{\Omega}R(q+\theta)\mbox{div} u dx\\[2mm]
&+\frac{1}{2}\Big(\|\sqrt{\rho}u_t\|^2+\|q_t\|^2+c_{v}\|\sqrt{\rho}\theta_t\|+\|H_t\|^2\Big)\\[2mm]
&\leq C\|\mbox{div}u\|^2+\delta_1 \mathcal{D}^2_1(t).
\end{split}
\end{equation}
Then, integrating (\ref{above-1}) over $[0,t]$, then  Lemma \ref{lem-1}, (\ref{u}) and (\ref{H}) implies Lemma \ref{lem-ut}.
\end{proof}

Now, we prove estimates of the first-order  derivatives $\nabla q$.

\begin{lem}\label{2}
\begin{equation}\label{q}
	\|(\nabla q,\mbox{div}u,\nabla\theta)\|^2+c\int^t_0\|(\nabla\mbox{div}u,\Delta\theta)\|^2ds\leq \|(\nabla q,\mbox{div}u,\nabla\theta)\|^2(0)+C\delta_1\int^t_0\mathcal{D}_1^2(s)ds,
\end{equation}	
where $c>0$, $C>0$ are positive constants independent of $t$.	Moreover,
\begin{equation}\label{q-1}
	\int^t_0\|\nabla^2\theta\|^2ds\leq C\mathcal{E}^2_1(0)+C\delta_1\int^t_0\mathcal{D}_1^2(s)ds.
\end{equation}	
\end{lem}
\begin{proof}
Computing the following integral
\begin{equation*}
\int_{\Omega}\left\{(\ref{prob-linear})_2\cdot\nabla\mbox{div}u+\nabla(\ref{prob-linear})_1\cdot R\nabla q+\nabla(\ref{prob-linear})_3\cdot \nabla\theta \right\} dx,
\end{equation*}
and integrating by parts,  noting that $\int_{\Omega}\mbox{curl}^2u\cdot\nabla \mbox{div}udx=0$ thanks to $\mbox{curl}u\times n|_{\partial\Omega}=0$,
one has
\begin{equation*}
\begin{split}
&\frac{1}{2}\frac{d}{dt}\int_{\Omega} \Big(\rho|\mbox{div}u|^2+R|\nabla q|^2+\rho|\nabla\theta|^2\Big)dx
+(2\mu+\lambda)\|\nabla\mbox{div}u\|^2+\kappa\|\Delta\theta\|^2\\[2mm]
&=\frac{1}{2}\int_{\Omega}q_{t}|\mbox{div}u|^2dx+\frac{c_{v}}{2}\int_{\Omega}q_{t}|\nabla \theta|^2dx-\int_{\Omega}u_t\cdot\nabla q\mbox{div}udx\\[2mm]
&+\int_{\Omega}\Big(\rho u\cdot\nabla u+R\nabla(q\theta)-\mbox{curl}H\times H\Big)\cdot\nabla\mbox{div}udx\\[2mm]
&-\int_{\Omega}\nabla\mbox{div}(qu)\cdot R\nabla qdx-c_{v}\int_{\Omega}\Big(\nabla q(\theta_{t}+u\cdot \nabla\theta)+\rho\nabla(u\cdot\nabla\theta)\Big)\cdot\nabla\theta dx\\[2mm]
&-\int_{\Omega}\nabla\Big(R(\rho\theta+q)\mbox{div}u\Big)\cdot\nabla\theta dx +\int_{\Omega}\nabla\Big(\lambda|\mbox{div}u|^2+2\mu|Su|^2+\eta|\mbox{curl}H|^2\Big)\cdot\nabla\theta dx.
\end{split}
\end{equation*}
The nonlinear linear terms on the right hand-side can be easily controlled by $\mathcal{E}_1(t)\mathcal{D}^2_1(t)$. Therefore, integrating over $[0,t]$ implies the desired estimate (\ref{q}). Furthermore, applying Lemma \ref{lem-elliptic-3} to $\theta$, we get
\begin{equation*}
\begin{split}
\int_0^t\|\nabla^2\theta\|^2ds&\leq C\int_0^t\Big(\|\Delta\theta\|^2+\|\nabla\theta\|^2\Big)ds\\[2mm]
&\leq C\mathcal{E}^2_1(0)+C\delta_1\int_0^t\mathcal{D}_1^2(s)ds,
\end{split}
\end{equation*}
due to Lemma \ref{lem-1}  and (\ref{q}). Thus the proof of Lemma \ref{2} is
completed.
\end{proof}	

In order to obtain the dissipation estimates of $\|\nabla^2 u\|^2$ and $\|\nabla^2 H\|^2$, in view of Lemma  \ref{lem-elliptic-1}, we need to  control $\|\mbox{curl}^2u\|_{L^2_t(L^2)}$ and $\|\mbox{curl}^2H\|_{L^2_t(L^2)}$.

\begin{lem}\label{curl^2u}
	\begin{equation}\label{2H0}
		\|(\sqrt{\rho}\mbox{curl}u,\mbox{curl}H)\|^2+c\int^t_0\|(\mbox{curl}^2u,\mbox{curl}^2H)\|^2ds
\leq C\|(\sqrt{\rho}\mbox{curl}u,\mbox{curl}H)\|^2(0)+C\delta_1\int^t_0\mathcal{D}_1^2(s)ds,
	\end{equation}	
	where $c>0$, $C>0$ are positive constants independent of $t$.	 Moreover,
\begin{equation}\label{2H}
		\int^t_0\Big(\|\nabla ^2u\|^2+\|\nabla^2 H\|^2\Big)ds
\leq C\mathcal{E}^2_1(0)+C\delta_1\int^t_0\mathcal{D}_1^2(s)ds.
	\end{equation}	
\end{lem}
\begin{proof}Let $\mbox{curl}u=w$, then one has the  identity
\begin{equation}\label{identity1}
\mbox{curl}(u\cdot\nabla)u =u\cdot\nabla w-w\cdot\nabla u+w\mbox{div}u.
\end{equation}
Taking operator $\mbox{curl}$ to equation $(\ref{prob-linear})_2$, we easily derive the following equations for $w$ :
\begin{equation}\label{curl-equation-u}
	 \rho(w_{t}+ u\cdot\nabla w)-\mu\Delta w=K+\mbox{curl}(\mbox{curl}H\times H),
\end{equation}
where
\begin{equation*}
	K=-\rho w\cdot\nabla u-\rho w\mbox{div}u-\nabla q\times(u_t+u\cdot\nabla u).
\end{equation*}
Multiplying (\ref{curl-equation-u}) by $w$, noting that $\Delta w=-\mbox{curl curl}w$, integrating by parts with $w\times n|_{\partial \Omega}=0$,  we obtain
\begin{equation}\label{curlw}
\frac{1}{2}\frac{d}{dt}\int_{\Omega}\rho|w|^2dx+\mu\|\mbox{curl}w\|^2
=\int_{\Omega}K\cdot w dx+\int_{\Omega}\mbox{curl}(\mbox{curl}H\times H)\cdot wdx.
\end{equation}
We estimate the right-hand side as follows:
\begin{equation*}
\begin{split}
\int_{\Omega}K\cdot wdx&\leq C\Big(\|w\|_{L^6}\|w\|\|\nabla u\|_{L^3}+\|\nabla q\|_{L^3}\|u_{t}\|_{L^6}\|w\|
+\|\nabla q\|_{L^6}\|w\|_{L^6}\|u\|_{L^6}\|\nabla u\|\Big)\\[2mm]
&\leq  \frac{\mu}{4}\|\mbox{curl}w\|^2+C\delta_1 \Big(\|\nabla u_{t}\|^2+\|\nabla u\|^2\Big),
\end{split}
\end{equation*}
where we have used the estimate $\|w\|_{L^6}\leq C\|\nabla w\|^2\leq C(\|\mbox{curl}w\|^2+\|w\|^2)$ thanks to Lemma \ref{lem-div-curl} with $w\times n|_{\partial\Omega}=0$.
Integrating by parts with the boundary condition $w\times n|_{\partial\Omega}=0$ again implies
\begin{equation*}
\begin{split}
\int_{\Omega}\mbox{curl}(\mbox{curl}H\times H)\cdot wdx&=\int_{\Omega}(\mbox{curl}H\times H)\cdot \mbox{curl} wdx\\[2mm]
&\leq \|H\|_{L^{\infty}}\|\mbox{curl}H\|\|\mbox{curl}w\|\\[2mm]
&\leq \frac{\mu}{4}\|\mbox{curl}w\|^2+C\delta_1 \|\mbox{curl}H\|^2.
\end{split}
\end{equation*}
Putting these into (\ref{curlw}) yields
\begin{equation}\label{2-2}
\frac{1}{2}\frac{d}{dt}\int_{\Omega}\rho|w|^2dx+\frac{\mu}{2}\|\mbox{curl}w\|^2
\leq C\delta_1 \Big(\|\nabla u_{t}\|^2+\|\nabla u\|^2+\|\nabla H\|^2\Big).
\end{equation}
Similarly, letting  $\mbox{curl}H=\phi$,   taking $\mbox{curl}$ to equation $(\ref{prob-linear})_{4}$, we obtain
\begin{equation}\label{curl-equation-H}
		\phi_{t}+ \eta\mbox{curl}^2\phi=\mbox{curl}^2(u\times H).
	\end{equation}
 Multiplying (\ref{curl-equation-H}) by $\phi$, integrating by parts with $\phi\times n|_{\partial\Omega}= 0$,  one has
\begin{equation*}
\begin{split}
		\frac{1}{2}\frac{d}{dt}\int_{\Omega}|\phi|^2dx+ \eta\|\mbox{curl}\phi\|^2&=\int_{\Omega}\mbox{curl}^2(u\times H)\cdot \phi dx\\[2mm]
         &=\int_{\Omega}\mbox{curl}(u\times H)\cdot \mbox{curl}\phi dx\\[2mm]
         &\leq \frac{\eta}{2}\|\mbox{curl}\phi\|^2+C\|\mbox{curl}(u\times H)\|^2\\[2mm]
         &\leq\frac{\eta}{2}\|\mbox{curl}\phi\|^2+C\delta_1\Big(\|\nabla u\|^2+\|\nabla H\|^2\Big).
         \end{split}
	\end{equation*}
Therefore,
\begin{equation}\label{2-3}
\frac{1}{2}\frac{d}{dt}\int_{\Omega}|\phi|^2dx+ \frac{\eta}{2}\|\mbox{curl}\phi\|^2\leq C\delta_1\Big(\|\nabla u\|^2+\|\nabla H\|^2\Big).
        \end{equation}
Summing up (\ref{2-2}) and (\ref{2-3}),  integrating over $[0,t]$ gives (\ref{2H0}).

Furthermore, applying Lemma \ref{lem-elliptic-1} with $v=H$, and noting that $\Delta H=-\mbox{curl}^2 H$, we get
     \begin{equation*}
\begin{split}
		\int_0^t\|\nabla^2H\|^2 ds&\leq C\int_0^t\Big(\|\Delta H\|^2+\|\nabla H\|^2\Big)ds\\[2mm]
&\leq C\int_0^t\Big(\|\mbox{curl}^2 H\|^2+\|\nabla H\|^2\Big)ds\\[2mm]
&\leq C\mathcal{E}^2_1(0)+C\delta_1\int^t_0\mathcal{D}_1^2(s)ds,
         \end{split}
	\end{equation*}
thanks to the basic energy estimate Lemma \ref {lem-1} and (\ref{2H0}).

Similarly, applying Lemma \ref{lem-elliptic-1} with $v=u$, we obtain
     \begin{equation*}
\begin{split}
		\int_0^t\|\nabla^2u\|^2ds&\leq C\int_0^t\Big(\|\Delta u\|^2+\|\nabla u\|^2\Big)ds\\[2mm]
&\leq C\int_0^t\Big(\|\mbox{curl}^2 u\|^2+\|\nabla \mbox{div}u\|^2+\|\nabla u\|^2\Big)ds\\[2mm]
&\leq C\mathcal{E}^2_1(0)+C\delta_1\int^t_0\mathcal{D}_1^2(s)ds,
         \end{split}
	\end{equation*}
thanks to Lemma \ref {lem-1}, Lemma \ref{2} and (\ref{2H0}).
        \end{proof}

Next, we prove  estimates of the first-order temporal derivatives $(q_t,u_t,\theta_t, H_t)$.
\begin{lem}\label{6-1}
\begin{equation*}
	\|(q_t,u_t,\theta_t,H_t)\|^2+c\int^t_0\|(\nabla u_t,\nabla\theta_t,\nabla H_t)\|^2ds\leq C\|(q_t,u_t,\theta_t,H_t)(0)\|^2+ C\delta_1\int^t_0\mathcal{D}_1^2(s)ds,
\end{equation*}	
where $c>0$, $C>0$ are positive constants independent of $t$.	
\end{lem}
\begin{proof}
Computing the following integral
\begin{equation*}
\int_{\Omega}\left\{\partial_{t}(\ref{prob-linear})_1Rq_{t}+\partial_{t}(\ref{prob-linear})_2\cdot u_{t}+\partial_{t}(\ref{prob-linear})_3\theta_{t}+\partial_{t}(\ref{prob-linear})_4\cdot H_{t}\right\} dx,
\end{equation*}
and noting that  differentiation of the system (\ref{prob-linear}) with respect to $t$ will keep the boundary conditions (\ref{boundary-linear}),
it is easy to  the prove Lemma \ref{6-1}, we omit it.
\end{proof}

\subsection{Estimates for the second-order derivatives}

Firstly, we prove  estimates of the second-order derivatives $(\nabla u_t,\nabla\theta_t, \nabla H_t)$.
\begin{lem}\label{lem-utt}
	\begin{equation*}
		\|(\nabla u_t,\nabla\theta_t,\nabla H_t)\|^2+c\int^t_0\|(u_{tt},q_{tt},\theta_{tt},H_{tt})\|^2ds\\[2mm]
		\leq C\mathcal{E}^2_1(0)+C\delta_1\int^t_0\mathcal{D}_1^2(s)ds,
	\end{equation*}	
	 where $c>0$, $C>0$ are positive constants independent of $t$.	
\end{lem}

\begin{proof}
	Computing the following integral
	\begin{equation*}
\int_{\Omega}\left\{\partial_t(\ref{prob-linear})_2\cdot u_{tt}+\partial_t(\ref{prob-linear})_1\cdot  q_{tt}+\partial_t(\ref{prob-linear})_3\cdot \theta_{tt}+\partial_t(\ref{prob-linear})_4\cdot H_{tt}\right\}dx
\end{equation*}
integrating by parts, one has
\begin{equation}\label{above-2}
\begin{split}	
&\frac{1}{2}\frac{d}{dt}\Big\{\mu\|\mbox{curl}u_t\|^2+(2\mu+\lambda)\|\mbox{div}u_t\|^2+\kappa\|\nabla\theta_t\|^2+\eta\|\mbox{curl}H_t\|^2\Big\}
-\frac{d}{dt}\int_{\Omega}R(q_t+\theta_t)\mbox{div} u_t dx\\[2mm]
&+\|\sqrt{\rho}u_{tt}\|^2+\|q_{tt}\|^2+c_{v}\|\sqrt{\rho}\theta_{tt}\|^2+\|H_{tt}\|^2\\[2mm]
&=\int_{\Omega}\Big((R+1) q_{tt}+2R\theta_{tt}\Big)\mbox{div}u_t dx
+\int_{\Omega}\mbox{div}(qu)_t q_{tt} dx\\[2mm]
&+\int_{\Omega}\Big(-q_t(u_t+u\cdot\nabla u)+\rho (u\cdot\nabla u)_{t}-R\nabla (q\theta)_{t}+(\mbox{curl}H\times H)_{t}\Big)\cdot u_{tt}dx\\[2mm]
&+\int_{\Omega}\Big(-q_t(\theta_t+u\cdot\nabla \theta)-c_{v}\rho (u\cdot\nabla \theta)_{t}\Big) \theta_{tt}dx\\[2mm] &+\int_{\Omega}\Big(\lambda(\mbox{div}u)^2+2\mu|Su|^2-R(\rho\theta+q)\mbox{div}u+\eta|\mbox{curl}H|^2\Big)_{t}\theta_{tt}dx
+\int_{\Omega}\mbox{curl}(u\times H)_{t}\cdot H_{tt}dx\\[2mm]
&\leq \frac{1}{2}\Big(\|\sqrt{\rho}u_{tt}\|^2+\|q_{tt}\|^2+c_{v}\|\sqrt{\rho}\theta_{tt}\|+\|H_{tt}\|^2\Big)+C\|\mbox{div}u_{t}\|^2
+C\mathcal{E}_1(t) \mathcal{D}^2_1(t).
\end{split}
\end{equation}
On the other hand, applying Proposition \ref{prop-imp} to { $u_t$ and $H_t$ with $u_t\cdot n |_{\partial\Omega}=0$, $H_t\cdot n|_{\partial\Omega}=0$}, it holds
 \begin{equation}\label{above}
 \|\nabla u_{t}\|^2\leq C\big(\|\mbox{curl}u_t\|^2+\|\mbox{div}u_t\|^2\big), \quad\mbox{and}
 \quad\|\nabla H_{t}\|^2\leq C\big\|\mbox{curl}H_t\|^2.
 \end{equation}
 Then, integrating (\ref{above-2}) over $[0,t]$,  recalling Lemma \ref{6-1}, and using (\ref{above})
 we prove Lemma \ref{lem-utt}.
\end{proof}

Next, we prove  estimates of the second-order derivatives $\nabla q_t$.
\begin{lem}\label{lem-uxxt}
	\begin{equation}\label{6-1-0}
		\|(\nabla q_t,\sqrt{\rho}\mbox{div}u_t,\sqrt{\rho}\nabla\theta_t)\|^2+c\int^t_0\|(\nabla\mbox{div}u_t,\Delta\theta_t)\|^2ds\leq \|(\nabla q_t,\sqrt{\rho}\mbox{div}u_t,\sqrt{\rho}\nabla\theta_t)\|^2(0)+C\delta_1\int^t_0\mathcal{D}_1^2(s)ds,
	\end{equation}	
	where $c>0$, $C>0$ are positive constants independent of $t$.	 Moreover,
\begin{equation}\label{2theta_t}
	\int^t_0\|\nabla^2\theta_t\|^2ds\leq C\mathcal{E}_1^2(0)+C\delta_1\int^t_0\mathcal{D}_1^2(s)ds.
	\end{equation}
\end{lem}

\begin{proof}
	Taking $\partial_{t}$ to $(\ref{prob-linear})_2$, multiplying the resulted identity by $\nabla\mbox{div}u_t$, and  integrating  over $\Omega$, one has
\begin{equation*}
\begin{split}
		&\frac{1}{2}\frac{d}{dt}\int_{\Omega}\rho|\mbox{div}u_t|^2dx+(2\mu+\lambda)\|\nabla\mbox{div}u_t\|^2-R\int_{\Omega}(\nabla q_t+\nabla\theta_t)\cdot\nabla\mbox{div}u_tdx\\[2mm]
&=\frac{1}{2}\int_{\Omega}q_t|\mbox{div}u_t|^2dx+\int_{\Omega}u_{tt}\cdot\nabla q\mbox{div}u_tdx\\[2mm]
&+\int_{\Omega}\Big(\rho (u\cdot\nabla u)_t+q_t(u_t+u\cdot\nabla u)+R\nabla(q\theta)_t-(\mbox{curl}H\times H)_{t}\Big)\cdot \nabla \mbox{div}u_tdx\\[2mm]
&\leq \|q_t\|_{L^{\infty}}\|\mbox{div}u_t\|^2+\|u_{tt}\|\|\nabla q\|_{L^3}\|\mbox{div}u_t\|_{L^6}\\[2mm]
&+\Big(\|(u,q,\theta,H\|_2+\|q_t\|_{1}+\|u\|_{2}\|q_t\|_{1}\Big)\cdot\Big(\|\nabla\mbox{div}u_t\|^2+\|\nabla u_t\|^2+\|\nabla^2u\|^2+
\|\nabla q_t\|^2+\|\nabla \theta_t\|^2+\|\nabla H_t\|^2\Big),
\end{split}
	\end{equation*}
which yields
\begin{equation}\label{new-1}
\begin{split}
		&\frac{1}{2}\frac{d}{dt}\int_{\Omega}\rho|\mbox{div}u_t|^2dx+\frac{2\mu+\lambda}{2}\|\nabla\mbox{div}u_t\|^2-R\int_{\Omega}(\nabla q_t+\nabla\theta_t)\cdot\nabla\mbox{div}u_tdx\\[2mm]
&\leq \delta_1\Big(\|\mbox{div}u_t\|^2+\|u_{tt}\|^2+
\|\nabla\mbox{div}u_t\|^2+\|\nabla u_t\|^2+\|\nabla^2u\|^2+
\|\nabla q_t\|^2+\|\nabla \theta_t\|^2+\|\nabla H_t\|^2\Big).
\end{split}
	\end{equation}
 To eliminate the singular terms on the left-hand side of (\ref{new-1}), we take $\partial_{t}$ to $(\ref{prob-linear})_3$, multiply the resulted identity by $\Delta \theta_t$,
and integrate over $\Omega$, then it holds
\begin{equation*}
\begin{split}
		&\frac{c_{v}}{2}\frac{d}{dt}\int_{\Omega}\rho|\nabla \theta_t|^2dx+\kappa\|\Delta\theta_t\|^2
+R\int_{\Omega}\nabla \theta_t\cdot\nabla\mbox{div}u_tdx\\[2mm]
&=\frac{c_v}{2}\int_{\Omega}q_t|\nabla\theta_t|^2dx+c_{v}\int_{\Omega}\theta_{tt}\nabla q\cdot \nabla\theta_tdx
+\int_{\Omega}\Big(\rho(u\cdot\nabla\theta)_t+q_t(\theta_t+u\cdot\nabla\theta)\Big)\Delta\theta_{t}dx\\[2mm]
&-\int_{\Omega}\partial_{t}\Big(\lambda(\mbox{div}u)^2+2\mu |S(u)|^2-R(\rho\theta+q)\mbox{div}u+\eta|\mbox{curl}H|^2\Big)\Delta\theta_{t}dx\\[2mm]
&\leq \|q_t\|_{L^3}\|\nabla\theta_t\|_{L^6}\|\nabla\theta_t\|+\|\nabla q\|_{L^3}\|\nabla\theta_{t}\|_{L^6}\|\theta_{tt}\|\\[2mm]
&+\Big(\|(u,\theta, H)\|_{2}+\|(q,q_t)\|_{1}\Big)\cdot\Big(\|\Delta\theta_{t}\|^2+\|\nabla u_t\|^2+\|\nabla \theta_t\|^2+\|\nabla q_t\|^2+\|\nabla\mbox{div} u_t\|^2+\|\nabla q_t\|^2+\|\nabla^2H_t\|^2\Big).
\end{split}
	\end{equation*}
Therefore, by using the elliptic estimate $\|\nabla^2\theta_t\|^2\leq C\Big( \|\Delta\theta_t\|^2+\|\nabla\theta_t\|^2\Big)$,  one has
\begin{equation}\label{new-2}
\begin{split}
		&\frac{c_{v}}{2}\frac{d}{dt}\int_{\Omega}\rho|\nabla \theta_t|^2dx+\frac{\kappa}{2}\|\Delta\theta_t\|^2
+R\int_{\Omega}\nabla \theta_t\cdot\nabla\mbox{div}u_tdx\\[2mm]
&\leq
C\delta_1\Big(\|\theta_{tt}\|^2+\|\nabla u_t\|^2+\|\nabla \theta_t\|^2+\|\nabla q_t\|^2+\|\nabla\mbox{div} u_t\|^2+\|\nabla q_t\|^2+\|\nabla^2H_t\|^2\Big).
\end{split}
	\end{equation}
Similarly, taking $\partial_{t}\nabla$ to $(\ref{prob-linear})_1$,  multiplying the resulted identity by $R\nabla q_t$,
and integrating over $\Omega$, we get
\begin{equation}\label{new-4}
		\frac{R}{2}\frac{d}{dt}\int_{\Omega}|\nabla q_t|^2dx+R\int_{\Omega}\nabla q_t\cdot\nabla\mbox{div}u_tdx
=R\int_{\Omega}\partial_{t}\nabla\mbox{div}(qu)\cdot \nabla q_tdx.
	\end{equation}
The integration on the right hand side will appear term $\nabla^2q_t$ which is not included in $\mathcal{D}_1(t)$, we shall work on it carefully.
\begin{equation*}
\begin{split}
&\int_{\Omega}\partial_{t}\nabla\mbox{div}(qu)\cdot \nabla q_tdx=\int_{\Omega}u\cdot\nabla^2q_t \cdot \nabla q_tdx\\[2mm]
&+\int_{\Omega}\Big(\mbox{div}u\nabla q_t+q_t\nabla\mbox{div}u+\mbox{div}u_t\nabla q+q \nabla\mbox{div}u_t+\nabla u_t\cdot \nabla q
+u_t\cdot\nabla^2q+\nabla u\cdot \nabla q_t\Big)\cdot \nabla q_tdx\\[2mm]
&\leq\int_{\Omega}u\cdot\nabla^2q_t \cdot \nabla q_tdx+\Big( \|\nabla u\|_{L^{\infty}}+ \|\nabla \mbox{div}u\|_{L^3}+\|\nabla q\|_{L^3}+\|q\|_{L^{\infty}}+\|\nabla^2 q\| \Big)\cdot
\Big(\|\nabla q_t\|^2+\|u_t\|^2_{2}\Big)\\[2mm]
&\leq \int_{\Omega}u\cdot\nabla^2q_t \cdot \nabla q_tdx + \delta_1 \Big(\|\nabla q_t\|^2+\|u_t\|^2_{2}\Big).
\end{split}
\end{equation*}
In addition,  integrating by parts with $u\cdot n|_{\partial\Omega}=0$ gives
\begin{equation*}
 \int_{\Omega}u\cdot\nabla^2q_t \cdot \nabla q_tdx=-\frac{1}{2}\int_{\Omega}\mbox{div}u|\nabla q_t|^2dx\leq\delta_1\|\nabla q_t\|^2.
\end{equation*}
Therefore, plugging    into (\ref{new-4}), we obtain
\begin{equation}\label{new-3}
\begin{split}
		\frac{R}{2}\frac{d}{dt}\int_{\Omega}|\nabla q_t|^2dx+R\int_{\Omega}\nabla q_t\cdot\nabla\mbox{div}u_tdx\leq \delta_1 \Big(\|\nabla q_t\|^2+\|u_t\|^2_{2}\Big).
\end{split}
	\end{equation}
Finally, summing  (\ref{new-1}),  (\ref{new-2}) and (\ref{new-3}), and integrating the resulting estimates over $[0,t]$, we obtain the desired estimate (\ref{6-1-0}).

In addition, applying Lemma \ref{lem-elliptic-3} to $\theta_t$, that is $\int_{0}^t\|\nabla^2\theta_t\|^2ds\leq C\int_0^t\big(\|\Delta\theta_t\|^2+\|\nabla\theta_t\|^2\big)ds$, so together
with Lemma \ref{6-1} and (\ref{6-1-0}), we obtain the estimate (\ref{2theta_t}). This
completes the proof of Lemma \ref{lem-uxxt}.
\end{proof}

\begin{lem}\label{curl^2u_t}
	\begin{equation}\label{2-0}
		\|(\sqrt{\rho}\mbox{curl}u_t,\mbox{curl}H_t)\|^2+c\int^t_0\|(\mbox{curl}^2u_t,\mbox{curl}^2H_t)\|^2ds\leq
\|(\sqrt{\rho}\mbox{curl}u_t,\mbox{curl}H_t)\|^2(0)+ C\delta_1\int^t_0\mathcal{D}_1^2(s)ds,
	\end{equation}	
	where $c>0$, $C>0$ are positive constants independent of $t$.	Moreover, it holds
\begin{equation}
\begin{split}
		\int_0^t\big(\|\nabla^2H_t\|^2 +\|\nabla^2 u_t\|^2\big)ds\leq  C\mathcal{E}^2_1(0)+C\delta_1\int^t_0\mathcal{D}_1^2(s)ds.
         \end{split}
	\end{equation}
\end{lem}
\begin{proof}
Notice that differentiation  of the equations  (\ref{curl-equation-u}) and (\ref{curl-equation-H}) with respect to $t$ will keep the boundary conditions $w_{t}\times n|_{\partial\Omega}=0$ and $\phi_{t}\times n |_{\partial\Omega}=0$.
Similar to  the proof on Lemma \ref{curl^2u}, computing the integral
	\begin{equation*}
\int_{\Omega}
\partial_t(\ref{curl-equation-u})\cdot w_tdx, \quad  \mbox{and}\quad  \int_{\Omega} \partial_t(\ref{curl-equation-H})\cdot\phi_tdx
\end{equation*}
gives
\begin{equation}\label{2-2-2}
\frac{1}{2}\frac{d}{dt}\int_{\Omega}\rho|w_t|^2dx+\frac{\mu}{2}\|\mbox{curl}w_t\|^2
\leq C \mathcal{E}_1(t)\mathcal{D}^2_1(t),
\end{equation}
and
\begin{equation}\label{2-3-2}
\frac{1}{2}\frac{d}{dt}\int_{\Omega}|\phi_t|^2dx+ \frac{\eta}{2}\|\mbox{curl}\phi_t\|^2\leq C \mathcal{E}_1(t)\mathcal{D}^2_1(t).
        \end{equation}
Then, summing up (\ref{2-2-2}) and (\ref{2-3-2}),  integrating over $[0,t]$ gives (\ref{2-0}).

 Furthermore, applying Lemma \ref{lem-elliptic-1} with $v=H_t$, and noting that $\Delta H_t=-\mbox{curl}^2 H_t$, we get
     \begin{equation*}
\begin{split}
		\int_0^t\|\nabla^2H_t\|^2 ds&\leq C\int_0^t\Big(\|\Delta H_t\|^2+\|\nabla H_t\|^2\Big)ds\\[2mm]
&\leq C\int_0^t\Big(\|\mbox{curl}^2 H_t\|^2+\|\nabla H_t\|^2\Big)ds\\[2mm]
&\leq C\mathcal{E}^2_1(0)+C\delta_1\int^t_0\mathcal{D}_1^2(s)ds,
         \end{split}
	\end{equation*}
thanks to Lemma \ref {6-1} and (\ref{2-0}).

Similarly, applying Lemma \ref{lem-elliptic-1} with $v=u_t$, we get
     \begin{equation*}
\begin{split}
		\int_0^t\|\nabla^2u_t\|^2ds&\leq C\int_0^t\Big(\|\Delta u_t\|^2+\|\nabla u_t\|^2\Big)ds\\[2mm]
&\leq C\int_0^t\Big(\|\mbox{curl}^2 u_t\|^2+\|\nabla \mbox{div}u_t\|^2+\|\nabla u_t\|^2\Big)ds\\[2mm]
&\leq C\mathcal{E}^2_1(0)+C\delta_1\int^t_0\mathcal{D}_1^2(s)ds,
         \end{split}
	\end{equation*}
thanks to Lemma \ref {lem-uxxt}, Lemma \ref{6-1} and (\ref{2-0}).
\end{proof}

Next, we shall derive the estimate of $(\nabla^2u, \nabla^2H, \nabla^2\theta)$ into the following two lemmas.

\begin{lem}\label{3}
	\begin{equation}\label{Deltatheta}
		\|(\nabla\mbox{div}u,\Delta\theta)\|^2+c\int^t_0\|(\nabla q_t,\nabla\theta_t)\|^2ds\leq C\|(\nabla\mbox{div}u,\Delta\theta)\|^2(0)
+ C\delta_1\int^t_0\mathcal{D}_1^2(s)ds,
	\end{equation}	
	where $c>0$, $C>0$ are positive constants independent of $t$.	

Moreover,
\begin{equation}\label{2theta}
\|\nabla^2\theta\|^2\leq  C\mathcal{E}^2_1(0)+ C\delta_1\int^t_0\mathcal{D}_1^2(s)ds.
\end{equation}
\end{lem}
\begin{proof}
Taking $\partial_t$ to equation $(\ref{prob-linear})_2$,  multiplying the resulted equality by $\nabla\mbox{div}u$ in $L^{2}(\Omega)$, and using the fact
$\int_{\Omega}\mbox{curl}^2u_t\cdot\nabla\mbox{div}udx=0$,
we get
\begin{equation}\label{3-1}
	\begin{split}
&\frac{2\mu+\lambda}{2}\frac{d}{dt}\int_{\Omega}|\nabla\mbox{div}u|^2dx-R\int_{\Omega}(\nabla q_t+\nabla \theta_t)\cdot\nabla \mbox{div}udx\\[2mm]
&=\int_{\Omega}\rho (u_{tt}+u\cdot \nabla u_t)\cdot\nabla \mbox{div}udx+\int_{\Omega}\Big(q_t u_t+(\rho u)_t\cdot\nabla u
+R\nabla(q\theta)_{t}-(\mbox{curl} H\times H)_{t}\Big)\cdot \nabla \mbox{div}udx\\[2mm]
&\leq \mathcal{E}_1(t)\mathcal{D}_1^2(t).
\end{split}
\end{equation}
To eliminate the singular terms on the left-hand side of (\ref{3-1}), we apply operator $\nabla$ to equations $(\ref{prob-linear})_1$ and $(\ref{prob-linear})_3$, then multiply the resulting equality by $R\nabla q_t$ and $\nabla\theta_t$ in $L^{2}(\Omega)$ to get
\begin{equation}\label{3-2}
	\begin{split}
&R\|\nabla q_t\|^2
+R\int_{\Omega}\nabla\mbox{div}u\cdot\nabla q_t dx\\[2mm]
&=-R\int_{\Omega}\nabla(\mbox{div}(qu))\cdot\nabla q_t dx \\[2mm]
&\leq \mathcal{E}_1(t)\mathcal{D}_1^2(t),
\end{split}
\end{equation}
and
\begin{equation}\label{3-3}
	\begin{split}
&\frac{\kappa}{2}\frac{d}{dt}\int_{\Omega}|\Delta\theta|^2dx+c_v\|\sqrt{\rho}\nabla\theta_t\|^2
+R\int_{\Omega}\nabla\mbox{div}u\cdot\nabla\theta_tdx\\[2mm]
&=\int_{\Omega}\Big(c_v\nabla q\cdot (\theta_t+u\cdot\theta)+c_v\rho\nabla(u\cdot\nabla\theta)\Big)\cdot\nabla\theta_tdx\\[2mm]
&+\int_{\Omega}\nabla \Big(\lambda(\mbox{div}u)^2+2\mu |S(u)|^2-R(\rho\theta+q)\mbox{div}u+\eta|\mbox{curl}H|^2\Big)\cdot \nabla\theta_tdx\\[2mm]
&\leq \mathcal{E}_1(t)\mathcal{D}_1^2(t).
\end{split}
\end{equation}
Then, from (\ref{3-1}), combining with (\ref{3-2}) and (\ref{3-3}), we obtain (\ref{Deltatheta}).

Moreover, applying Lemma \ref{lem-elliptic-3} to $\theta$, that is $\|\nabla^2\theta\|^2\leq C\big(\|\Delta\theta\|^2+\|\nabla\theta\|^2\big)$, so
combining with Lemma \ref{2} and (\ref{Deltatheta}), we obtain the estimate (\ref{2theta}). The proof is completed.
\end{proof}	

\begin{lem}\label{7}
	\begin{equation}\label{7-1}
		\|(\sqrt{\rho}\mbox{curl}^2u,\mbox{curl}^2H)\|^2+c\int^t_0\|(\mbox{curl}^3 u,\mbox{curl}^3 H)\|^2ds\leq \|(\sqrt{\rho}\mbox{curl}^2u,\mbox{curl}^2H)\|^2(0)+C\delta_1\int^t_0\mathcal{D}_1^2(s)ds,
	\end{equation}	
	where $c>0$, $C>0$ are positive constants independent of $t$.	
Moreover,
\begin{equation}\label{7-3}
\|\nabla^2H\|^2+\int_0^t\|\nabla^3H\|^2\leq
C\mathcal{E}^2_1(0)+C\delta_1\int^t_0\mathcal{D}_1^2(s)ds.
\end{equation}
and
\begin{equation}\label{7-2}
\|\nabla^2 u\|^2+\int_0^t\|\nabla\mbox{curl}^2u\|^2\leq
C\mathcal{E}^2_1(0)+C\delta_1\int^t_0\mathcal{D}_1^2(s)ds.
\end{equation}
\end{lem}
\begin{proof}

Firstly,   integrating by parts with $w_t\times n|_{\partial\Omega}=0$, it holds
\begin{equation*}
\begin{split}
&\int_{\Omega}\rho w_t\cdot\mbox{curl}^2wdx
=\int_{\Omega}\mbox{curl}(\rho w_t)\cdot\mbox{curl}wdx\\[2mm]
&=\frac{1}{2}\frac{d}{dt}\int_{\Omega}\rho|\mbox{curl}w|^2dx-\frac{1}{2}\int_{\Omega}q_t|\mbox{curl}w|^2dx
+\int_{\Omega}(\nabla q\times w_t)\cdot \mbox{curl}wdx.
\end{split}
\end{equation*}
Taking
(\ref{curl-equation-u}) inner with $\mbox{curl}^2 w$, and integrating over $\Omega$, we obtain
\begin{equation}\label{5-1}
\begin{split}
&\frac{1}{2}\frac{d}{dt}\int_{\Omega}\rho|\mbox{curl}w|^2 dx+\mu \|\mbox{curl}^2w\|^2\\[2mm]
&=\frac{1}{2}\int_{\Omega}q_t|\mbox{curl}w|^2dx
-\int_{\Omega}(\nabla q\times w_t)\cdot \mbox{curl}w dx+\int_{\Omega}\Big(K+\mbox{curl}(\mbox{curl}H\times H)-\rho u\cdot w\Big)\cdot \mbox{curl}^2w dx\\[2mm]
&\leq \mathcal{E}_1(t)\mathcal{D}_1^2(t).
\end{split}
\end{equation}
Similarly, taking
$(\ref{curl-equation-H})$ inner with $\mbox{curl}^2\phi$ in $L^{2}(\Omega)$, one has
	\begin{equation}\label{5-2}
\begin{split}
&\frac{1}{2}\frac{d}{dt}\int_{\Omega}|\mbox{curl}\phi|^2dx+\eta \|\mbox{curl}^2\phi\|^2
=\int_{\Omega}\mbox{curl}^2(u\times H)\cdot \mbox{curl}^2\phi dx\\[2mm]
&\leq \mathcal{E}_1(t)\mathcal{D}_1^2(t).
\end{split}
\end{equation}
Therefore, integrating (\ref{5-1})  and (\ref{5-2}) over $[0,t]$ gives (\ref{7-1}).

Moreover, since $\mbox{curl}^2u\cdot n|_{\partial \Omega}=0$, applying Lemma \ref{lem-div-curl}, it holds
\begin{equation*}
\begin{split}
\int_0^t\|\nabla\mbox{curl}^2u\|^2ds&\leq C\int_0^t\Big(\|\mbox{curl}^3u\|^2+\|\mbox{curl}^2u\|^2\Big)ds\\[2mm]
&\leq C\mathcal{E}^2_1(0)+C\delta_1\int^t_0\mathcal{D}_1^2(s)ds,
\end{split}
\end{equation*}
due to Lemma \ref{curl^2u} and (\ref{7-1}).

Noting that $\mbox{curl}^2H\cdot n|_{\partial \Omega}=0$, applying elliptic estimate and  Lemma \ref{lem-div-curl} again, we has
\begin{equation*}
\begin{split}
\int_0^t\|\nabla^3H\|^2ds&\leq C\int_0^t\Big(\|\Delta H\|_1^2+\|\nabla H\|^2\Big)ds\\[2mm]
&\leq  C\int_0^t\Big(\|\mbox{curl}^2H\|_1^2+\|\nabla H\|^2\Big)ds\\[2mm]
&\leq  C\int_0^t\Big(\|\mbox{curl}^3H\|^2+\|\mbox{curl}^2H\|^2+\|\nabla H\|^2\Big)ds\\[2mm]
&\leq C\mathcal{E}^2_1(0)+C\delta_1\int^t_0\mathcal{D}_1^2(s)ds,
\end{split}
\end{equation*}
thanks to Lemma \ref{lem-1}, Lemma \ref{curl^2u} and (\ref{7-1}).

Furthermore, applying elliptic estimate and the fact $\mbox{div} H=0$, it holds
\begin{equation*}
\begin{split}
\|\nabla^2H\|^2&\leq C\Big(\|\Delta H\|^2+\|\nabla H\|^2\Big)\\[2mm]
&\leq C\Big(\|\mbox{curl}^2 H\|^2+\|\nabla H\|^2\Big)\\[2mm]
&\leq C\mathcal{E}^2_1(0)+C\delta_1\int^t_0\mathcal{D}_1^2(s)ds,
\end{split}
\end{equation*}
thanks to Lemma \ref{lem-ut} and (\ref{7-1}).  Applying elliptic estimate  again, it holds
\begin{equation*}
\begin{split}
\|\nabla^2u\|^2&\leq C(\|\Delta u\|^2+\|\nabla u\|^2\Big)\\[2mm]
&\leq C(\|\mbox{curl}^2 u\|^2+\|\nabla\mbox{div}u\|^2+\|\nabla u\|^2\Big)\\[2mm]
&\leq C\mathcal{E}^2_1(0)+C\delta_1\int^t_0\mathcal{D}_1^2(s)ds,
\end{split}
\end{equation*}
thanks to  Lemma \ref{lem-ut},  Lemma \ref{3} and (\ref{7-1}).
\end{proof}

Now, we prove the estimate of $\nabla^2q$.
\begin{lem}\label{lem-9}
\begin{equation}\label{red}
	\|\nabla^2 q\|^2+c\int_{0}^{t}\|\nabla^2\mbox{div}u\|^2ds
\leq C\mathcal{E}^2_1(0)+ \varepsilon\int_{0}^t \|\nabla^2 q\|^2ds
+C\delta_1\int_0^t \mathcal{D}^2_1(s)ds.
\end{equation}
\end{lem}
\begin{proof}
Taking $\partial_{x_i}$ to $(\ref{prob-linear})_2$, multiplying the resulting identities by $\partial_{x_i}\nabla\mbox{div}u$, integrating over $\Omega$, we obtain
\begin{equation*}
	\begin{split}
		&(2\mu+\lambda)\|\nabla^2\mbox{div}u\|^2-R\int_{\Omega}(\nabla^2q+\nabla^2\theta)\cdot \nabla^2\mbox{div}u dx\\[2mm]
&=-\mu \int_{\Omega}\nabla \mbox{curl}^2u\cdot \nabla^2\mbox{div}u dx+\int_{\Omega}\Big(\nabla (\rho(u_t+u\cdot\nabla u))+R\nabla^2(q\theta)-\nabla(\mbox{curl}H\times H\Big)\cdot \nabla^2\mbox{div}u dx.
\end{split}
\end{equation*}
To eliminate the singular terms on the left-hand side of the above, we
compute the  integral	
$$\int_{\Omega}\left\{\partial_{x_i}\nabla(\ref{prob-linear})_1\cdot R(\partial_{x_i}\nabla q+\partial_{x_i}\nabla\theta)\right\}dx$$ to obtain
\begin{equation*}
	\begin{split}
		&\frac{R}{2}\frac{d}{dt}\int_{\Omega}|\nabla^2 q|^2 dx+R\frac{d}{dt}\int_{\Omega}\nabla^2 q\cdot\nabla^2\theta dx +R\int_{\Omega}(\nabla^2q+\nabla^2\theta)\cdot \nabla^2\mbox{div}u dx\\[2mm]
&=R\int_{\Omega}\nabla^2 q\cdot\nabla^2\theta_t dx
-R\int_{\Omega}\nabla^2\mbox{div}(qu)\cdot(\nabla^2q+\nabla^2\theta)dx,
		\end{split}
\end{equation*}
Summing up the above  yields
\begin{equation}\label{right}
	\begin{split}
		&\frac{R}{2}\frac{d}{dt}\int_{\Omega}|\nabla^2 q|^2 dx +R\frac{d}{dt}\int_{\Omega}\nabla^2 q\cdot\nabla^2\theta dx +(2\mu+\lambda)\|\nabla^2\mbox{div}u\|^2\\[2mm]
&=-\mu\int_{\Omega}\nabla \mbox{curl}^2u\cdot \nabla^2\mbox{div}u dx+\int_{\Omega}\Big(\nabla (\rho(u_t+u\cdot\nabla u))+R\nabla^2(q\theta)-\nabla(\mbox{curl}H\times H)\Big)\cdot \nabla^2\mbox{div}udx\\[2mm]
&+R\int_{\Omega}\nabla^2 q\cdot\nabla^2\theta_t dx
-R\int_{\Omega}(\nabla^2q+\nabla^2\theta)\cdot\nabla^2\mbox{div}(qu) dx\\[2mm]
&\equiv K_1+K_2+K_3+K_4.
		\end{split}
\end{equation}
Using Cauchy's inequality,
\begin{equation}\label{K1}
K_1\leq \frac{2\mu+\lambda}{4}\|\nabla^2\mbox{div}u\|^2+C\|\nabla \mbox{curl}^2u\|^2,
\end{equation}
\begin{equation}\label{K2}
\begin{split}
K_2&\leq \frac{2\mu+\lambda}{4}\|\nabla^2\mbox{div}u\|^2+C\Big(\|\nabla (\rho(u_t+u\cdot\nabla u))\|^2+\|\nabla^2(q\theta)\|^2+\|\nabla(\mbox{curl}H\times H)\|^2\Big)\\[2mm]
&\leq \frac{2\mu+\lambda}{4}\|\nabla^2\mbox{div}u\|^2+\Big(\|(u,q,\theta,H)\|^2_{2}\Big) \Big(\|\nabla^2u\|^2+\|\nabla u_t\|^2+\|\nabla^2q\|^2+\|\nabla^2\theta\|^2+\|\nabla\mbox{curl}H\|^2\Big)\\[2mm]
&\leq\frac{2\mu+\lambda}{4}\|\nabla^2\mbox{div}u\|^2+ \delta\Big(\|\nabla^2u\|^2+\|\nabla u_t\|^2+\|\nabla^2q\|^2+\|\nabla^2\theta\|^2+\|\nabla\mbox{curl}H\|^2\Big),
\end{split}
\end{equation}
and
\begin{equation}\label{K3}
K_3\leq \varepsilon \|\nabla^2 q\|^2+C_{\varepsilon}\|\nabla^2\theta_t\|^2.
\end{equation}
The  term $K_4$ becomes
\begin{equation*}
\begin{split}
K_4=R\int_{\Omega}(\nabla^2q+\nabla^2\theta)\cdot \Big(\nabla^2q \mbox{div}u+2\nabla q\cdot\nabla\mbox{div}u+q\nabla^2\mbox{div}u
+\nabla^2u\cdot\nabla q+u\cdot\nabla^3q+2\nabla u\cdot\nabla^2q\Big)dx
\end{split}
\end{equation*}
The term  $\nabla^3q$ including on $K_4$, by using integration by parts with $u\cdot n|_{\partial\Omega}=0$, we obtain
\begin{equation*}
\begin{split}
&R\int_{\Omega}(\nabla^2q+\nabla^2\theta)u\cdot\nabla^3q dx\\[2mm]
&=-\frac{R}{2}\int_{\Omega}\mbox{div}u |\nabla^2q|^2 dx-R\int_{\Omega}\nabla^3\theta u\cdot\nabla^2q dx
-R\int_{\Omega}\mbox{div}u\nabla^2\theta\cdot \nabla^2q dx\\[2mm]
&\leq \Big(\|\mbox{div}u\|_{L^{\infty}}+\|u\|_{L^{\infty}}\Big)\cdot\Big(\|\nabla^2q\|^2+\|\nabla^2\theta\|^2+\|\nabla^3\theta\|^2\Big)\\[2mm]
&\leq \delta \Big(\|\nabla^2q\|^2+\|\nabla^2\theta\|^2+\|\nabla^3\theta\|^2\Big).
\end{split}
\end{equation*}
While the other terms on $K_4$ is controlled by $\delta \Big(\|\nabla^2q\|^2+\|\nabla^2\theta\|^2+\|\nabla^3u\|^2\Big)$. So
\begin{equation}\label{K4}
K_4\leq \delta \Big(\|\nabla^2q\|^2+\|\nabla^2\theta\|^2+\|\nabla^3\theta\|^2+\|\nabla^3u\|^2\Big).
\end{equation}
Putting estimates (\ref{K1})-(\ref{K4})into (\ref{right}) implies
\begin{equation}
	\begin{split}
		&\frac{R}{2}\frac{d}{dt}\int_{\Omega}|\nabla^2 q|^2 dx+R\frac{d}{dt}\int_{\Omega}\nabla^2 q\cdot\nabla^2\theta dx +\frac{2\mu+\lambda}{2}\|\nabla^2\mbox{div}u\|^2 \\[2mm]
&\leq C\|\nabla \mbox{curl}^2u\|^2+ \varepsilon\|\nabla^2 q\|^2+C_{\varepsilon}\|\nabla^2\theta_t\|^2+\delta_1 \Big(\|\nabla^2q\|^2+\|\nabla^2\theta\|^2+\|\nabla^3\theta\|^2+\|\nabla^3u\|^2\Big).
\end{split}
\end{equation}
 Integrating the above over $[0,t]$ yields
 \begin{equation*}
	\begin{split}
		&\frac{R}{4}\|\nabla^2 q\|^2 +\frac{2\mu+\lambda}{2}\int_0^t\|\nabla^2\mbox{div}u\|^2 ds
\leq C\|\nabla^2\theta\|^2+\varepsilon\int_0^t \|\nabla^2 q\|^2 ds\\[2mm]
&+ C\int_0^t\Big(\|\nabla \mbox{curl}^2u\|^2+\|\nabla^2\theta_t\|^2\Big)ds+
\delta_1 \int_0^t \Big(\|\nabla^2q\|^2+\|\nabla^2\theta\|^2+\|\nabla^3\theta\|^2+\|\nabla^3u\|^2\Big)ds,
\end{split}
\end{equation*}
 which together (\ref{2theta}), (\ref{2theta_t}) and (\ref{7-2}) implies the desired result (\ref{red}).

\end{proof}

At last, with  Lemma \ref{lem-1}-Lemma \ref{lem-9}  at hand, we can use Lemma \ref{lem-elliptic-1} to obtain the energy estimates of $\|(\nabla^3 u,\nabla^3 \theta, \nabla^3 H)\|$. Moreover,  we complete the proof of Proposition \ref{prop}  by adding Lemma \ref{lem-1}-Lemma \ref{lem-9},  and using equation $(\ref{prob-linear})_2$ to obtain the dissipation  estimates of $\|\nabla^2 u\|_{1}^2+\|\nabla q\|_1^2$ by choosing $\varepsilon$ small in Lemma
\ref{lem-9}.  Here, we omit the detail.

\section{Smooth solution }
In this section, we shall prove Theorem \ref{thm-2},  that is to  improve the regularity of the strong solutions in Theorem \ref{thm-1}.
Let
\begin{equation}\label{E2}
\mathcal{E}(t)=\|(u,\theta, H)\|_{4}+\|q\|_{3}+\|(q_t,u_t, H_t,\theta_t)\|_2+\|(q_{tt},u_{tt},\theta_{tt},H_{tt})\|
\end{equation}
and
\begin{equation}\label{D2}
\mathcal{D}(t)=\|(\nabla u,\nabla\theta,\nabla H)\|_{3}+\|\nabla q\|_{2}+\|(q_t,u_{t},H_{t})\|_{2}+\|\theta_{t}\|_{3}
+\|(u_{tt},\theta_{tt},H_{tt})\|_{1}+\|q_{tt}\|.
\end{equation}
We show a priori estimates of $H^4$ norm in the following.
\begin{prop}\label{prop-2} ({\bf a priori estimates of $H^4$ norm})
	Let $(q,u,\theta, H)$ be a solution to the initial boundary value problem (\ref{prob-linear}), (\ref{in}) and (\ref{boundary-linear})in  $t\in[0,T]$. Then there exists  positive constants $C$ and $\delta$ which are independent of $t$,  such that if
	\begin{equation*}
	\sup_{0\leq t\leq T}\mathcal{E}(t)\leq \delta,
	\end{equation*}
	then there holds, for any $t\in[0,T]$,
	\begin{equation*}
	\mathcal{E}^2(t)+C\int_{0}^{t}\mathcal{D}^2(s)ds\leq C\mathcal{E}^2(0).
	\end{equation*}
\end{prop}

We shall divide the
proof of Proposition \ref{prop-2}  into several lemmas.
Firstly, with the results in Proposition \ref{prop} and  using elliptic estimates, one can obtain the following  dissipation estimates.

\begin{lem}\label{lem4-3-1}
\begin{equation}
\begin{split}
\int_{0}^{t}\|\nabla^2\mbox{curl}^2u\|^2ds
\leq C\mathcal{E}^2_1(0),
\end{split}
\end{equation}
\begin{equation}\label{theta4}
\begin{split}
\int_{0}^{t}\|\nabla^4\theta\|^2ds
\leq C\mathcal{E}^2_1(0),
\end{split}
\end{equation}
and
\begin{equation}\label{H4}
\begin{split}
\int_{0}^{t}\|\nabla^4H\|^2ds
\leq C\mathcal{E}^2_1(0).
\end{split}
\end{equation}

\end{lem}
\begin{proof}
Applying Lemma \ref{lem-div-curl} with $v=\mbox{curl}^{2}u$,
\begin{equation}\label{u4}
\begin{split}
\|\nabla^2\mbox{curl}^2u\|^2&\leq \|\mbox{curl}^{3}u\|^2_{1}+\|\mbox{curl}^{2}u\|^2+|\mbox{curl}^{2}u\cdot n|_{H^{3/2}(\partial\Omega)}\\[2mm]
&\leq \|\mbox{curl}^{3}u\|^2_{1}+ \|\mbox{curl}^{2}u\|^2,
\end{split}
\end{equation}
thanks to $\mbox{curl}^{2}u\cdot n|_{\partial\Omega}=0$.
Moreover, noting that
\begin{equation*}
\|\mbox{curl}^{3}u\|^2_{1}=\|\mbox{curl}^{2}w\|^2_{1}=\|\Delta w\|^2_{1},
\end{equation*}
and using the equation of $w$ (\ref{curl-equation-u}), we obtain
\begin{equation*}
\begin{split}
\|\Delta w\|^2_{1}&\leq \|\rho w_{t}+\rho (u\cdot\nabla w-w\cdot\nabla u)+\nabla q\times (u_t+u\cdot\nabla u)+\rho w\mbox{div}u
-H\cdot\nabla \mbox{curl}H+\mbox{curl}H\cdot\nabla H\|^2_{1}\\[2mm]
&\leq C\|w_t\|_{1}^2+\Big(\|\nabla q\|_1^2+\|u\|_{3}^2+\|H\|_2^2\Big)\cdot\Big(\|\nabla w_t\|^2+\|\nabla w\|_1^2+\|u_t\|_{2}^2+\|\nabla u\|_2^2+\|\nabla H\|_{2}^2\Big).
\end{split}
\end{equation*}
Plugging the above back in (\ref{u4}), one gets
\begin{equation*}
\begin{split}
\int_{0}^{t}\|\nabla^2\mbox{curl}^2u\|^2ds
&\leq  C\int_{0}^{t}\Big(\|w_t\|_{1}^2+\|\mbox{curl}^{2}u\|^2\Big)ds+\delta\int_{0}^{t}\Big(\|\nabla w_t\|^2+\|\nabla w\|_1^2+\|u_t\|_{2}^2+\|\nabla u\|_2^2+\|\nabla H\|_{2}^2\Big)ds\\[2mm]
&\leq (C+\delta)\int_{0}^{t}\mathcal{D}^2_1(s)ds
\leq C\mathcal{E}^2_1(0).
\end{split}
\end{equation*}
For  magnetic field $H$, applying Lemma \ref{lem-elliptic-1} and using equation $(\ref{prob-linear})_{4}$, we get
\begin{equation*}
\begin{split}
\int_{0}^{t}\|\nabla^4H\|^2ds&\leq\int_{0}^{t}\Big(\|\Delta H\|_2^2+ \|\nabla H\|^2\Big)ds\\[2mm]
&\leq \int_{0}^{t}\Big(\| H_t\|_2^2+ \|\mbox{curl}(u\times H)\|_2^2+\|\nabla H\|^2\Big)ds\\[2mm]
&\leq \int_{0}^{t}\Big(\| H_t\|_2^2+ \| H\|^2_{2}\|\nabla u\|^2_{2}+\| u\|^2_{2}\|\nabla H\|^2_{2}+\|\nabla H\|^2\Big)ds\\[2mm]
&\leq (C+\delta_1) \int_{0}^{t}\mathcal{D}^2_1(s)ds\leq C\mathcal{E}^2_1(0).
\end{split}
\end{equation*}

Finally, we estimate the four-order derivatives of $\theta$. Rewrite equation $(\ref{prob-linear})_2$ as
\begin{equation}\label{N}
\kappa\Delta\theta=c_{v}\rho(\theta_{t}+ u\cdot\nabla\theta)+R\mbox{div}u-\lambda(\mbox{div}u)^2-2\mu |S(u)|^2+R(\rho\theta+q)\mbox{div}u-\eta|\mbox{curl}H|^2\equiv N,
\end{equation}
then, applying the elliptic estimates  of $\theta$, $\theta_t$ (Lemma \ref{lem-elliptic-3}), it is easy to obtain the estimates of $\|\nabla^4\theta\|_{L^{2}_t(L^2)}$ and $\|\nabla^3\theta_t\|_{L^{2}_t(L^2)}$.  Precisely,
\begin{equation*}
\begin{split}
\int_{0}^{t}\|\nabla^4\theta\|^2ds&\leq \int_{0}^{t}\|N\|^2_{2}ds+\int_{0}^{t}\|\nabla \theta\|^2ds\\[2mm]
&\leq (C+\delta) \int_{0}^{t}\mathcal{D}^2_1(s)ds\leq C\mathcal{E}^2_1(0).
\end{split}
\end{equation*}

The proof of Lemma \ref{lem4-3-1} is completed.
\end{proof}

Next,  we prove estimates of the second-order temporal derivatives $(q_{tt},u_{tt},\theta_{tt},H_{tt})$.

\begin{lem}\label{lem-uxtt}
	\begin{equation}\label{uxtt}
	\|(q_{tt},u_{tt},\theta_{tt},H_{tt})\|^2+c\int^t_0\|(\nabla u_{tt},\nabla\theta_{tt},\nabla H_{tt})\|^2ds
\leq C\mathcal{E}^2(0).
\end{equation}	
where $c>0$, $C>0$ are positive constants independent of $t$.	
Moreover,
\begin{equation}\label{Ht}
\|(u_t,\theta_t,H_t)\|^2_{2}\leq C\mathcal{E}^2(0),
\end{equation}
and
\begin{equation}\label{theta_t}
\begin{split}
\int_{0}^{t}\|\nabla^3\theta_t\|^2ds
\leq C\mathcal{E}^2(0).
\end{split}
\end{equation}

\end{lem}

\begin{proof}
Computing the following integral
\begin{equation*}
\int_{\Omega}\left\{\partial_{tt}(\ref{prob-linear})_1Rq_{tt}+\partial_{tt}(\ref{prob-linear})_2\cdot u_{tt}+\partial_{tt}(\ref{prob-linear})_3\theta_{tt}+\partial_{tt}(\ref{prob-linear})_4\cdot H_{tt}\right\} dx,
\end{equation*}
and noting that  differentiation of the system (\ref{prob-linear}) with respect to $t$ will keep the boundary conditions (\ref{boundary-linear}),
after integrating by parts, one has
	\begin{equation}\label{tt}
\begin{split}
&\frac{1}{2}\frac{d}{dt}\int_{\Omega} \Big(Rq_{tt}^2+\rho|u_{tt}|^2+c_{v}\rho\theta_{tt}^2+|H_{tt}|^2\Big)dx
+\mu\|\mbox{curl}u_{tt}\|^2+(2\mu+\lambda)\|\mbox{div}u_{tt}\|^2+\kappa\|\nabla\theta_{tt}\|^2+\eta\|\mbox{curl}H_{tt}\|^2\\[2mm]
&=-R\int_{\Omega}(\mbox{div}(qu))_{tt}q_{tt}dx-\int_{\Omega}\Big(q_{tt}(u_t+u\cdot \nabla u)+2q_{t}(u_t+u\cdot \nabla u)_{t}-\rho u_{tt}\cdot\nabla u+
2\rho u_{t}\cdot\nabla u_{t}\Big)\cdot u_{tt}dx\\[2mm]
&-R\int_{\Omega}(\nabla (q\theta))_{tt}\cdot u_{tt}dx+\int_{\Omega}(\mbox{curl}H\times H)_{tt}\cdot u_{tt}dx\\[2mm]
&-c_{v}\int_{\Omega}\Big(q_{tt}(\theta_t+u\cdot \nabla \theta)+2q_{t}(\theta_t+u\cdot \nabla \theta)_{t}-\rho u_{tt}\cdot\nabla \theta+
2\rho \theta_{t}\cdot\nabla \theta_{t}\Big) \theta_{tt}dx\\[2mm]
&+\int_{\Omega}\Big(\lambda(\mbox{div}u)^2+2\mu |S(u)|^2-R(\rho\theta+q)\mbox{div}u+\eta|\mbox{curl}H|^2\Big)_{tt} \theta _{tt}dx\\[2mm]
&+\int_{\Omega}(\mbox{curl}(u\times H))_{tt}\cdot H_{tt}dx.
\end{split}
\end{equation}
The term involving third-order derivations of $q$ can be estimated as follows:
\begin{equation}
\begin{split}
&-R\int_{\Omega}(\nabla (q\theta))_{tt}\cdot u_{tt}dx=R\int_{\Omega}(q\theta)_{tt}\cdot \mbox{div}u_{tt}dx\\[2mm]
&\leq \frac{2\mu+\lambda}{4}\|\mbox{div}u_{tt}\|^2+C\|(q\theta)_{tt}\|^2\\[2mm]
&\leq  \frac{2\mu+\lambda}{4}\|\mbox{div}u_{tt}\|^2+ \Big(\|(\theta,q)\|_2+\|q_t\|_1\Big)\Big(\|q_{tt}\|^2+\|\theta_{tt}\|^2+\|\nabla \theta_{t}\|^2\Big)\\[2mm]
&\leq  \frac{2\mu+\lambda}{4}\|\mbox{div}u_{tt}\|^2+C\delta_1 \mathcal{D}^2_1(t),
\end{split}
\end{equation}
and
\begin{equation}
\begin{split}
&\int_{\Omega}(\mbox{div}(qu))_{tt}q_{tt}dx=\int_{\Omega}u\cdot \nabla q_{tt} q_{tt}dx+\int_{\Omega}(\mbox{div}u_{tt}q+\mbox{div}u_{t}q_t+\mbox{div}uq_{tt})q_{tt}dx\\[2mm]
&=-\frac{1}{2}\int_{\Omega}\mbox{div}u q_{tt}dx+\int_{\Omega}(\mbox{div}u_{tt}q+\mbox{div}u_{t}q_t+\mbox{div}uq_{tt})q_{tt}dx\\[2mm]
&\leq \frac{2\mu+\lambda}{4}\|\mbox{div}u_{tt}\|^2+ \Big(\|\mbox{div}u\|_{L^{\infty}}+\|q\|_{L^{\infty}}+\|q_t\|_{L^{3}}\Big)\Big(\|\mbox{div}u_{tt}\|^2+\|\nabla\mbox{div}u_t\|^2+\|q_{tt}\|^2\Big)\\[2mm]
&\leq \frac{2\mu+\lambda}{4}\|\mbox{div}u_{tt}\|^2+C\delta_1 \mathcal{D}^2_1(t).
\end{split}
\end{equation}
Noting that \begin{equation*}
(u\times H)\times n|_{\partial \Omega}=0,
\end{equation*}
due to the vector identity $(u\times H)\times n=(u\cdot n) H-(H\cdot n) u$ and $u\cdot n|_{\partial\Omega}=0$, $H\times n|_{\partial\Omega}=0$, and moreover,  differentiation  with respect  $t$ will keep the boundary conditions, it holds
\begin{equation*}
(u\times H)_{tt}\times n|_{\partial \Omega}=0.
\end{equation*}
Then the last term on the right hand-side on (\ref{tt}) becomes
\begin{equation}
\begin{split}
\int_{\Omega}(\mbox{curl}(u\times H))_{tt}\cdot H_{tt}dx&=\int_{\Omega}(u\times H)_{tt}\cdot\mbox{curl} H_{tt}dx \\[2mm]
&\leq \frac{\eta}{4}\|\mbox{curl} H_{tt}\|^2+C\|(u\times H)_{tt}\|^2\\[2mm]
&\leq \frac{\eta}{4}\|\mbox{curl} H_{tt}\|^2+C\delta_1\mathcal{D}^2_1(t).
\end{split}
\end{equation}
The other terms on the right hand-side on (\ref{tt}) can be estimated by $\delta_1\mathcal{D}^2_1(t)$ by using Holder inequality and Sobolev inequalities directly. Therefore,  (\ref{tt}) yields
\begin{equation}\label{tt-1}
\begin{split}
&\frac{1}{2}\frac{d}{dt}\int_{\Omega} \Big(Rq_{tt}^2+\rho|u_{tt}|^2+c_{v}\rho\theta_{tt}^2+|H_{tt}|^2\Big)dx\\[2mm]
&+\mu\|\mbox{curl}u_{tt}\|^2+\frac{2\mu+\lambda}{2}\|\mbox{div}u_{tt}\|^2+\kappa\|\nabla\theta_{tt}\|^2+\frac{\eta}{2}\|\mbox{curl}H_{tt}\|^2\\[2mm]
&\leq C\delta_1\mathcal{D}^2_1(t).
\end{split}
\end{equation}
Finally,  integrating  (\ref{tt-1}) over $[0,t]$ and applying Proposition \ref{prop-imp} to $u_{tt}$ and $H_{tt}$, we obtain
\begin{equation*}
\begin{split}
	&\|(q_{tt},u_{tt},\theta_{tt},H_{tt})\|^2+c_{11}\int^t_0\|(\nabla u_{tt},\nabla\theta_{tt},\nabla H_{tt})\|^2ds\\[2mm]
&\leq C\|(q_{tt},u_{tt},\theta_{tt},H_{tt})(0)\|^2+ C\delta_1\int^t_0\mathcal{D}^2_1(s)ds\\[2mm]
&\leq C\mathcal{E}^2(0),
\end{split}
\end{equation*}
where we have used Proposition \ref{prop} in the last inequality.

Furthermore, applying Lemma \ref{lem-elliptic-1} to $u_t$, $\theta_t$ and $H_t$, it is easy to see (\ref{Ht}).

Moreover, recalling (\ref{N}), we see that
\begin{equation*}
\begin{split}
\int_{0}^{t}\|\nabla^3\theta_t\|^2ds&\leq \int_{0}^{t}\|N_t\|^2_{1}ds+\int_{0}^{t}\|\nabla \theta_t\|^2ds\\[2mm]
&\leq  \int_{0}^{t}\|\nabla \theta_{tt}\|^2ds+(C+\delta) \int_{0}^{t}\mathcal{D}^2_1(s)ds\\[2mm]
&\leq C\mathcal{E}^2(0).
\end{split}
\end{equation*}
This complete the proof of Lemma \ref{lem-uxtt}.
\end{proof}

In order to obtain the estimate of $\|\nabla^2q_t\|_{L_t^2(L^2)}$,  we use the equation $q_t+\mbox{div}u+\mbox{div}(qu)=0$, we shall prove
$\|\nabla^3 q\|_{L_t^2(L^2)}$ and
$\|\nabla^{2}\mbox{div}u\|_{L_t^2(L^2)}$ (see Lemma \ref{lem-9}). To obtain $\|\nabla^3 q\|_{L_t^2(L^2)}$,  by using the elliptic estimate of Stokes-type equation $(\ref{prob-linear})_{2}$. However,  applying Lemma \ref{lem-elliptic-2} to control $\|\nabla^3 q\|_{L_t^2(L^2)}$, we shall estimate $\|\mbox{div}u\|_{L_t^{2}(H^{3})}$ firstly.

\begin{lem}\label{lem-q3}

\begin{equation}\label{q3}
	\begin{split}
		\frac{R}{4}\|\nabla^3 q\|^2+\frac{2\mu+\lambda}{2}\int_0^t\|\nabla^3\mbox{div}u\|^2ds
\leq C\mathcal{E}^2(0)+\varepsilon\int_{0}^t\|\nabla^3 q\|^2ds
+\delta \int_0^t\Big(\|\nabla^3q\|^2+\|\nabla^4u\|^2\Big)ds.
\end{split}
\end{equation}

\end{lem}

\begin{proof}
Taking $\partial_{x_i}\partial_{x_j}$ to $(\ref{prob-linear})_2$, multiplying the resulting identities by $\partial_{x_i}\partial_{x_j}\nabla\mbox{div}u$, integrating over $\Omega$, we obtain
\begin{equation*}
	\begin{split}
		&(2\mu+\lambda)\|\nabla^3\mbox{div}u\|^2-R\int_{\Omega}(\nabla^3q+\nabla^3\theta)\cdot \nabla^3\mbox{div}u dx\\[2mm]
&=-\mu \int_{\Omega}\nabla^2 \mbox{curl}^2u\cdot \nabla^3\mbox{div}u dx+\int_{\Omega}\Big(\nabla^2 (\rho(u_t+u\cdot\nabla u))+R\nabla^2(q\theta)-\nabla(H\cdot\nabla H-\frac{1}{2}\nabla|H|^2)\Big)\cdot \nabla^3\mbox{div}u dx
\end{split}
\end{equation*}
Computing the  integral	
$\int_{\Omega}\left\{\partial_{x_i}\partial_{x_j}\nabla(\ref{prob-linear})_1\cdot R(\partial_{x_i}\partial_{x_j}\nabla q+\partial_{x_i}\partial_{x_j}\nabla\theta)\right\}dx$, one has
\begin{equation*}
	\begin{split}
		&\frac{R}{2}\frac{d}{dt}\int_{\Omega}|\nabla^3 q|^2dx+R\frac{d}{dt}\int_{\Omega}\nabla^3 q\cdot\nabla^3\theta  dx +R\int_{\Omega}(\nabla^3q+\nabla^3\theta)\cdot \nabla^3\mbox{div}u dx\\[2mm]
&=R\int_{\Omega}\nabla^3 q\cdot\nabla^3\theta_t dx
-R\int_{\Omega}\nabla^3\mbox{div}(qu)\cdot(\nabla^3q+\nabla^3\theta)dx,
		\end{split}
\end{equation*}
Summing up the above  yields
\begin{equation}\label{right-q3}
	\begin{split}
		&\frac{R}{2}\frac{d}{dt}\int_{\Omega}|\nabla^3 q|^2 dx+R\frac{d}{dt}\int_{\Omega}\nabla^3 q\cdot\nabla^3\theta dx +(2\mu+\lambda)\|\nabla^3\mbox{div}u\|^2\\[2mm]
&=-\mu\int_{\Omega}\nabla^2 \mbox{curl}^2u\cdot \nabla^3\mbox{div}u dx+\int_{\Omega}\Big(\nabla^2 (\rho(u_t+u\cdot\nabla u))+R\nabla^3(q\theta)-\nabla^2(\mbox{curl}H\times H)\Big)\cdot \nabla^3\mbox{div}u dx\\[2mm]
&+R\int_{\Omega}\nabla^3 q\cdot\nabla^3\theta_t dx
-R\int_{\Omega}(\nabla^3q+\nabla^3\theta)\cdot\nabla^3\mbox{div}(qu) dx\equiv J_1+J_2+J_3+J_4.
		\end{split}
\end{equation}

Using Cauchy's inequality,
\begin{equation}\label{J1}
J_1\leq \frac{2\mu+\lambda}{4}\|\nabla^3\mbox{div}u\|^2+C\|\nabla^2 \mbox{curl}^2u\|^2,
\end{equation}

\begin{equation}\label{J2}
\begin{split}
J_2&\leq \frac{2\mu+\lambda}{4}\|\nabla^3\mbox{div}u\|^2+C\Big(\|\nabla^2 (\rho(u_t+u\cdot\nabla u))\|^2+\|\nabla^3(q\theta)\|^2+\|\nabla^2(\mbox{curl}H\times H)\|^2\Big)\\[2mm]
&\leq \frac{2\mu+\lambda}{4}\|\nabla^3\mbox{div}u\|^2+ \delta\Big(\|\nabla q\|_1^2+\|\nabla^2 u_t\|^2+\|\nabla ^2u\|_1^2+\|\nabla^3q\|^2+\|\nabla^3\theta\|^2+\|\nabla^3H\|^2\Big)
 \end{split}
\end{equation}
and
\begin{equation}\label{J3}
J_3\leq \varepsilon \|\nabla^3 q\|^2+C_{\varepsilon}\|\nabla^3\theta_t\|^2.
\end{equation}
The  term $J_4$ becomes
\begin{equation*}
\begin{split}
J_4=R\int_{\Omega}(\nabla^3q+\nabla^3\theta)\cdot \Big(&\nabla^3q \mbox{div}u+3\nabla^2 q\cdot\nabla\mbox{div}u+3\nabla q\cdot\nabla^2\mbox{div}u
+q\nabla^3\mbox{div}u\\[2mm]
&+\nabla^3u\cdot\nabla q+3\nabla^2u\cdot\nabla^2q+3\nabla u\cdot\nabla^3q+u\cdot\nabla^4 q\Big)dx
\end{split}
\end{equation*}
The term  $\nabla^4q$ including on $J_4$, by using integration by parts with $u\cdot n|_{\partial\Omega}=0$, we obtain
\begin{equation*}
\begin{split}
&R\int_{\Omega}(\nabla^3q+\nabla^3\theta)u\cdot\nabla^4qdx\\[2mm]
&=-\frac{R}{2}\int_{\Omega}\mbox{div}u |\nabla^3q|^2dx-R\int_{\Omega}\nabla^4\theta u\cdot\nabla^3q dx
-R\int_{\Omega}\mbox{div}u\nabla^3\theta\cdot \nabla^3q dx\\[2mm]
&\leq \Big(\|\mbox{div}u\|_{L^{\infty}}+\|u\|_{L^{\infty}}\Big)\cdot\Big(\|\nabla^3q\|^2+\|\nabla^3\theta\|^2+\|\nabla^4\theta\|^2\Big)\\[2mm]
&\leq \delta \Big(\|\nabla^3q\|^2+\|\nabla^3\theta\|^2+\|\nabla^4\theta\|^2\Big).
\end{split}
\end{equation*}
While the other terms on $J_4$ is controlled by $\delta \Big(\|\nabla^3q\|^2+\|\nabla^3\theta\|^2+\|\nabla^4u\|^2\Big)$. So
\begin{equation}\label{J4}
J_4\leq \delta \Big(\|\nabla^3q\|^2+\|\nabla^3\theta\|^2+\|\nabla^4\theta\|^2+\|\nabla^4u\|^2\Big).
\end{equation}
Putting estimates (\ref{J1})-(\ref{J4}) into (\ref{right-q3}) implies
\begin{equation*}
	\begin{split}
		&\frac{R}{2}\frac{d}{dt}\int_{\Omega}|\nabla^3 q|^2 dx+R\frac{d}{dt}\int_{\Omega}\nabla^3 q\cdot\nabla^3\theta dx +\frac{2\mu+\lambda}{2}\|\nabla^3\mbox{div}u\|^2\\[2mm]
&\leq C\|\nabla^2 \mbox{curl}^2u\|^2+ \varepsilon\|\nabla^3 q\|^2+C_{\varepsilon}\|\nabla^3\theta_t\|^2\\[2mm]
&+\delta \Big(\|\nabla^3q\|^2+\|\nabla q\|_1^2+\|\nabla^2u\|_{1}^2+\|\nabla^3\theta\|^2+\|\nabla^4\theta\|^2+\|\nabla^4u\|^2+\|\nabla^2u_t\|^2+\|\nabla^3H\|^2\Big).
\end{split}
\end{equation*}
Then, integrating over $[0,t]$, and using Cauchy inequality and Proposition \ref{prop}, Lemma \ref{lem4-3-1} implies the desired estimate (\ref{q3}).

\end{proof}

 Finally, we use Lemma \ref{lem-elliptic-1} to obtain the energy estimates $\|(\nabla^4u, \nabla^4\theta,\nabla^4H)\|$. On the other hand, we apply the elliptic estimates of Stokes-type systems (Lemma \ref{lem-elliptic-2}) to obtain the estimates of $\big(\|\nabla^3q\|+\|\nabla^4u\|\big)_{L^{2}_t(L^2)}$. Therefore, the high-order energy is closed, i.e. Proposition \ref{prop-2} is proved.

{\bf {Proof of  Proposition \ref{prop-2} }}
Applying stand elliptic estimates (i.e. Lemma \ref{lem-elliptic-1}) to $u$, $\theta$ and $H$, it is easy  to get the energy estimates $\|(\nabla^4u, \nabla^4\theta,\nabla^4H)\|$.

Next, rewriting $(\ref{prob-linear})_2$ as the following Stokes-type systems
\begin{equation*}
\left\{
\begin{array}{llll}
\mbox{div}u=\mbox{div}u,\\[2mm]
-\mu\Delta u+R\nabla q=
-\rho\big(u_{t}+u\cdot\nabla u\big)+(\mu+\lambda)\nabla\mbox{div}u-R\nabla\theta-R\nabla(q\theta)+\mbox{curl}H\times H\equiv f
\end{array}
\right.
\end{equation*}
Applying Lemma \ref{lem-elliptic-2} to the above Stokes-type systems, one has
\begin{equation}\label{this}
	\begin{split}
		\int_{0}^{t}\Big(\|\nabla ^2u\|_{2}^2+\|\nabla q\|_{2}^2\Big)ds&\leq C\Big(\int_{0}^{t}\|f\|_2^2ds+\int_{0}^{t}\|\mbox{div}u\|_{3}^2ds\Big)\\[2mm]
&\leq (C+\delta)\int_0^t\mathcal{D}^2_1(s)ds+ C_{*}\int_{0}^{t}\|\nabla^3\mbox{div}u\|^2ds\\[2mm]
&\leq C\mathcal{E}^2_1(0)+C_{*}\int_{0}^{t}\|\nabla^3\mbox{div}u\|^2ds,
\end{split}
\end{equation}
where we have used the result of Proposition \ref{prop}.
Choosing $\varepsilon$ small in Lemma \ref{lem-q3}, such that $\varepsilon C_{*}\leq\frac{2\mu+\lambda}{4}$, then, form (\ref{this}), we have
\begin{equation}\label{this-2}
	\begin{split}
		\frac{R}{4}\|\nabla^3 q\|^2+\frac{2\mu+\lambda}{4}\int_0^t\|\nabla^3\mbox{div}u\|^2ds
\leq C\mathcal{E}^2(0)+
\delta \int_0^t\Big(\|\nabla^3q\|^2+\|\nabla^4u\|^2\Big)ds.
\end{split}
\end{equation}
Therefore, together (\ref{this}) with (\ref{this-2}), for $\delta$ small,
\begin{equation*}
	\int_{0}^{t}\Big(\|\nabla ^2u\|_{2}^2+\|\nabla q\|_{2}^2\Big)ds\leq
 C\mathcal{E}^2(0).
\end{equation*}
So, we close the energy estimates. This complete the proof of  Proposition \ref{prop-2}.
\qed

\section {Decay rate}

For the solution $(q,u,\theta, H)$ of Theorem \ref{thm-2}, one has
 \begin{equation*}\int_0^{\infty}\|(\nabla q,\nabla u,\nabla\theta,\nabla H)\|_1^2dt<\infty, \quad
\int_0^{\infty}\left|\frac{d}{dt}\|(\nabla q,\nabla u,\nabla\theta,\nabla H)\|_1^2\right|dt<\infty
  \end{equation*}
  where $q=\rho-1$, $\theta=\mathcal{T}-1$, it yields that
\begin{equation}\label{tinfty}
\|(\nabla q,\nabla u,\nabla\theta,\mbox{curl}H)\|_1\rightarrow 0,\quad \mbox{as}\quad t\rightarrow \infty.
\end{equation}
In this section, we will prove Theorem \ref{thm-3} by the the following lemmas.
Let
\begin{equation*}
	\begin{split}
	&X(0,\infty;\mathcal{E}_0):=\big\{(q,u,\theta, H)\big|
q\in C^{0}(0,\infty, H^{3}(\Omega))\cap C^{1}(0,\infty, H^{2}(\Omega)),\\[1mm]
&(u, \theta, H)\in C^{0}(0,\infty,H^{4}(\Omega))\cap C^{1}(0,\infty,H^{2}(\Omega)),
\mbox{ and} \\[2mm]
	&\mathcal{E}^2(t)+\int_0^t\mathcal{D}^2(s)ds\leq C\mathcal{E}_0^2
\big\}
	\end{split}
\end{equation*}
where $\mathcal{E}$ and $\mathcal{D}$ is same as $(\ref{E2})$ and $(\ref{D2})$. It is observed that the solution of IBVP (\ref{prob-linear})-(\ref{boundary-linear}) is sought in the space $X(0,\infty,\mathcal{E}_0)$, where $0\leq\mathcal{E}_0\leq C\delta$ and $\delta$ is given in Theorem \ref{thm-2}.

The first and most important key point to prove Theorem \ref{thm-3} is to obtain the large-time behavior of the time derivatives $(q_t,u_t,\theta_t,H_t)$.

\begin{lem}\label{lem4-1}
For every solution $(q,u,\theta,H)\in X(0, \infty; \mathcal{E}_0)$ of the problem (\ref{prob-linear})-(\ref{boundary-linear}), it exists a time $T_1>0$,  depending on $(q,u,\theta, H)$, such that
\begin{equation}\label{deacy1}
\|(q_t,u_t,\theta_t,H_t)\|\leq Ct^{-1/2}\quad \forall t\geq T_1.
\end{equation}
\end{lem}
\begin{proof}
Differentiate $(\ref{prob-linear})_1$ with respect to $t$, multiply the resulting identity by $Rq_{t}$ and integrate over $\Omega$, it follows
\begin{equation}\label{deacy-qt}
\begin{split}
&\frac{R}{2}\frac{d}{dt}\int_{\Omega}q_t^2dx+R\int_{\Omega}\mbox{div}u_tq_tdx=-R\int_{\Omega}\left(\mbox{div}(qu)\right)_{t}q_{t}dx
\\[2mm]
&=-R\int_{\Omega}q_t u_{t}\cdot \nabla q dx -R\int_{\Omega}q_tu\cdot\nabla q_{t} dx-R\int_{\Omega}q_t^{2}\mbox{div}u dx
-R\int_{\Omega}qq_{t}\mbox{div}u_{t} dx\\[2mm]
&\equiv \sum_{j=1}^{4} J_{i}.
\end{split}
\end{equation}
Utilizing Lemma \ref{sobolev inequ} and H\"{o}lder's inequality, for $\epsilon>0$ we deduce that
\begin{equation*}
|J_1|\leq \|u_{t}\|_{L^6}\|\nabla q\|_{L^3}\|q_{t}\|\leq \varepsilon \|\nabla u_{t}\|^2+C_{\varepsilon}\|\nabla q\|_1^2\|q_t\|^2,
\end{equation*}
\begin{equation*}
|J_4|\leq \|\mbox{div}u_{t}\|\|q\|_{L^{\infty}}\|q_{t}\|\leq \varepsilon \|\nabla u_{t}\|^2+C_{\varepsilon}\|\nabla q\|_1^2\|q_t\|^2.
\end{equation*}
Integration by parts with boundary condition $u\cdot n|_{\partial \Omega}=0$ yields
\begin{equation*}
J_2=\frac{R}{2}\int_{\Omega}q_{t}^2\mbox{div}u dx,
\end{equation*}
which implies that
\begin{equation*}
J_2+J_3=-\frac{R}{2}\int_{\Omega}q_{t}^2\mbox{div}u dx.
\end{equation*}
In view of equation $(\ref{prob-linear})_{1}$, it holds out that
\begin{equation*}
\begin{split}
J_2+J_3&=\frac{R}{2}\int_{\Omega}q_{t}^2 \left(q_t+q\mbox{div}u+u\cdot\nabla q\right)dx\\[2mm]
&=\frac{R}{2}\frac{d}{dt}\int_{\Omega}qq_{t}^2dx-R\int_{\Omega}qq_{t}q_{tt}dx
+\frac{R}{2}\int_{\Omega}q_{t}^2q\mbox{div}u dx+\frac{R}{2}\int_{\Omega}q_{t}^2u\cdot\nabla qdx\\[2mm]
&\leq\frac{R}{2}\frac{d}{dt}\int_{\Omega}qq_{t}^2dx-R\int_{\Omega}qq_{t}q_{tt}dx+\|\left(\nabla u,\nabla q\right)\|_2^2\|q_t\|^2.
\end{split}
\end{equation*}
For the term $\int_{\Omega}qq_{t}q_{tt}dx$, using $(\ref{prob-linear})_{1}$ again, one has
\begin{equation*}
\begin{split}
-R\int_{\Omega}qq_{t}q_{tt}dx&=R\int_{\Omega}qq_{t}\left(\mbox{div}u+\mbox{div}(qu)\right)_{t}dx\\[2mm]
&\leq \varepsilon \|\nabla u_t\|^2+C_{\varepsilon} \|\left(\nabla u,\nabla q\right)\|_2^2\|q_t\|^2.
\end{split}
\end{equation*}
Plugging all above inequalities into (\ref{deacy-qt}), we obtain
\begin{equation}\label{lem4-1-2}
\frac{R}{2}\frac{d}{dt}\int_{\Omega}q_t^2dx-\frac{R}{2}\frac{d}{dt}\int_{\Omega}qq_{t}^2dx+R\int_{\Omega}\mbox{div}u_tq_t dx
\leq \varepsilon \|\nabla u_{t}\|^2+C_{\varepsilon}\|(\nabla q,\nabla u)\|_2^2\|q_t\|^2.
\end{equation}
In a similar way we get  the following inequalities. From $(\ref{prob-linear})_{2}$, we have
\begin{equation}\label{lem4-1-u}
\begin{split}
\frac{1}{2}&\frac{d}{dt}\int_{\Omega}\rho |u_t|^2 dx +R\int_{\Omega}u_{t}\cdot \nabla q_{t}dx+R\int_{\Omega}u_{t}\cdot \nabla \theta_{t}dx+\mu \|\mbox{curl}u_{t}\|^2+(2\mu+\lambda)\|\mbox{div}u_{t}\|^2\\[2mm]
\leq& (\varepsilon+\|\nabla u\|_1^2) \left(\|\nabla u_{t}\|^2+\|\mbox{curl}H_{t}\|^2\right)\\[2mm]
&+C_{\varepsilon}\left(\|u_{t}\|_1^2+\|(\nabla q,\nabla u,\nabla\theta)\|_1^2+\|\mbox{curl}H\|_2^2\right)\cdot\left(\|q_t\|^2+\|u_{t}\|^2+\|\theta_{t}\|^2+\|H_t\|^2\right).
\end{split}
\end{equation}
From $(\ref{prob-linear})_{3}$, one has
\begin{equation}\label{lem4-1-theta}
\begin{split}
\frac{c_{v}}{2}&\frac{d}{dt}\int_{\Omega}\rho |\theta_t|^2dx+R\int_{\Omega}\theta_{t}\mbox{div}u_{t}dx+\kappa \|\nabla\theta_{t}\|^2\\[2mm]
\leq &\left(\varepsilon+\|\nabla u\|_1^2\right) \left(\|\nabla \theta_{t}\|^2+\|\nabla u_{t}\|^2+\|\mbox{curl}H_{t}\|^2\right)\\[2mm]
&+C_{\varepsilon}\left(\|\theta_{t}\|_1^2+\|(\nabla q,\nabla\theta)\|_1^2+\|(\nabla u,\mbox{curl}H\|_2^2\right)\cdot\Big(\|q_t\|^2+\|u_{t}\|^2+\|\theta_{t}\|^2\Big).
\end{split}
\end{equation}
And from $(\ref{prob-linear})_{4}$, it gives that
\begin{equation}\label{lem4-1-H1}
\begin{split}
\frac{1}{2}\frac{d}{dt}\int_{\Omega}|H_{t}|^2dx+\eta\|\mbox{curl}H_{t}\|^2
\leq \varepsilon  \left(\|\nabla u_{t}\|^2+\|\mbox{curl}H_{t}\|^2\right)
+C_{\varepsilon}\|(\nabla u,\mbox{curl}H)\|_1^2\|H_{t}\|^2.
\end{split}
\end{equation}
Summing up the estimates (\ref{lem4-1-2})-(\ref{lem4-1-H1}), and using Proposition \ref{prop-imp}, it implies that
\begin{equation}\label{lem4-1-H}
\begin{split}
&\frac{1}{2}\frac{d}{dt}\left\{R\|q_t\|^2+\|\sqrt{\rho}u_t\|^2+c_{v}\|\sqrt{\rho}\theta_t\|^2+\|H_{t}\|^2\right\}
-\frac{R}{2}\frac{d}{dt}\int_{\Omega}qq_t^2dx\\[2mm]
&+c_{*}\|\nabla u_{t}\|^2+\kappa \|\nabla\theta_{t}\|^2+\eta\|\mbox{curl}H_{t}\|^2\\[2mm]
&\leq \left(\varepsilon+\|\nabla u\|_1^2\right) \Big(\|\nabla \theta_{t}\|^2+\|\nabla u_{t}\|^2+\|\mbox{curl}H_{t}\|^2\Big)\\[2mm]
&+C_{\varepsilon}\left(\|u_{t}\|_1^2+\|\theta_{t}\|_1^2+\|(\nabla q,\nabla u,\mbox{curl}H)\|_2^2+\|\nabla\theta\|_1^2\right)\cdot\Big(\|q_t\|^2+\|u_{t}\|^2+\|\theta_{t}\|^2+\|H_t\|^2\Big),
\end{split}
\end{equation}
where $c_{*}>0$ is some positive constant depending on $\mu$, $\lambda$ and $C_{\Omega}$ as in Proposition \ref{prop-imp}. Define
\begin{equation*}
	\begin{split}
\phi(t)&\equiv \frac{1}{2}\left\{ R\|q_t\|^2+\|\sqrt{\rho}u_t\|^2+c_{v}\|\sqrt{\rho}\theta_t\|^2+\|H_{t}\|^2\right\}
-\frac{R}{2}\int_{\Omega}q q_t^2dx\\[2mm]
&=\frac{R}{2}\int_{\Omega}(1-q)q_t^2dx+\frac{1}{2}\left\{\|\sqrt{\rho}u_t\|^2+c_{v}\|\sqrt{\rho}\theta_t\|^2+\|H_{t}\|^2\right\},
\end{split}
\end{equation*}
then it is obvious that
\begin{equation*}
\|(q_t,u_t,\theta_t,H_t)\|^2\leq C\phi(t).
\end{equation*}
Moreover, because the fact (\ref{tinfty}) is true, there exists a positive constant $T_1>0$ such that
\begin{equation*}
\varepsilon+\|\nabla u\|_1^2\leq\frac{1}{2}\min(c_{*},\kappa,\eta),\quad \forall t\geq T_1.
\end{equation*}
Therefore,
\begin{equation*}
\phi'(t)\leq a(t)\phi(t),\quad \forall t\geq T_1,\\[1mm]
\end{equation*}
where $a(t)=C_{\epsilon}\left(\|u_{t}\|_1^2+\|\theta_{t}\|_1^2+\|(\nabla q,\nabla u,\mbox{curl}H)\|_2^2+\|\nabla\theta\|_1^2\right).$ We have
\begin{equation*}
\int_{0}^{\infty}\phi(t)dt<\infty, \quad \int_{0}^{\infty}a(t)dt<\infty.
\end{equation*}
due to $(q,u,\theta,H)\in X(0, \infty; \mathcal{E}_0)$, Then Lemma \ref{lem-deacy} (i) implies the assertion of Lemma \ref{lem4-1}.
\end{proof}

\begin{lem}\label{lem4-2}
For every solution $(q,u,\theta,H)\in X(0, \infty; \mathcal{E}_0)$ of the problem (\ref{prob-linear})-(\ref{boundary-linear}), it exists a time $T_2=T_2(q,u,\theta,H)\geq T_1>0$ such that
\begin{equation}\label{deacy2}
\|(\nabla u,\nabla \theta,\nabla H)\|^2\leq C\|(q_t,u_t,\theta_t, H_t)\| \quad \forall t\geq T_2,
\end{equation}
\begin{equation}\label{deacy3}
\|(\nabla^2 \theta,\nabla^2H)\|^2\leq C\|(q_t,u_t,\theta_t, H_t)\| \quad \forall t\geq T_2.
\end{equation}
\end{lem}
\begin{proof}
Multiplying
\begin{equation*}
q_{t}+\mbox{div}(\rho u)=0
\end{equation*}
 by $R(\theta+1)\ln (q+1)$ and integrating over $\Omega$, after integration by parts with the boundary condition $u\cdot n|_{\partial\Omega}=0$, we have
\begin{equation}\label{lem-4-2-1}
\begin{split}
-\int_{\Omega}R(\theta+1)u\cdot\nabla qdx&=-\int_{\Omega}R q_{t}(\theta+1)\ln (q+1)dx+\int_{\Omega}R \rho \ln(q+1)u\cdot\nabla\theta dx\\[2mm]
&\leq C\|q_t\|+C\|q\|_{L^3}\|u\|_{L^6}\|\nabla\theta\|,
\end{split}
\end{equation}
where we have used the fact
\begin{equation*}
	|\ln(q+1)-\ln1|\leq C|q|.
\end{equation*}
Let us rewrite $(\ref{prob-linear})_{2}$ in the following  form
\begin{equation*}
\rho\big(u_{t}+u\cdot\nabla u\big)-\mu\Delta u-(\mu+\lambda)\nabla\mbox{div}u+R(\theta+1)\nabla q+R\nabla\theta=-Rq\nabla\theta+\mbox{curl}H\times H,
\end{equation*}
multiplication this equation by $u$, and integration by parts indicate that
\begin{equation*}
\begin{split}
&\mu\|\mbox{curl}u\|^2+(2\mu+\lambda)\|\mbox{div}u\|^2
+\int_{\Omega}R(\theta+1)u\cdot\nabla q dx
-\int_{\Omega}R\mbox{div}u\theta dx-\int_{\Omega}(\mbox{curl}H\times H) \cdot u dx\\[2mm]
&=-\int_{\Omega}\rho u_{t}\cdot u dx-\int_{\Omega}\rho (u\cdot\nabla) u\cdot u dx-\int_{\Omega}Rqu\cdot\nabla\theta dx,
\end{split}
\end{equation*}
then by means of H$\ddot{o}$lder's inequality, it leads to
\begin{equation}\label{lem-4-2-2}
\begin{split}
&\mu\|\mbox{curl}u\|^2+(2\mu+\lambda)\|\mbox{div}u\|^2
+\int_{\Omega}R(\theta+1)u\cdot\nabla q
-\int_{\Omega}R\mbox{div}u\theta-\int_{\Omega}(\mbox{curl}H\times H) \cdot u\\[2mm]
&\leq C\|u_{t}\|+C\|(u,q)\|_{L^{3}}\|(\nabla u, \nabla \theta)\|^2.
\end{split}
\end{equation}
Similarly, multiplying $(\ref{prob-linear})_{3}$ by $\theta$, we have
\begin{equation*}
\begin{split}
&\kappa \|\nabla\theta\|^2+R\int_{\Omega}\mbox{div}u\theta dx\\[2mm]
&=\int_{\Omega}\left(
-c_{v}\rho(\theta_{t}+ u\cdot\nabla\theta)-R(\rho\theta+q)\mbox{div}u+\lambda(\mbox{div}u)^2+2\mu |S(u)|^2+\eta|\mbox{curl}H|^2\right)\theta dx,
\end{split}
\end{equation*}
which implies
\begin{equation}\label{lem-4-2-3}
\kappa \|\nabla\theta\|^2+R\int_{\Omega}\mbox{div}u\theta dx
\leq C\|\theta_{t}\|+C(\|\theta\|_{L^{\infty}} +\|(u,q,\theta)\|_{L^{3}})\|(\nabla u, \nabla \theta, \mbox{curl}H)\|^2.
\end{equation}
 And,  multiplying $(\ref{prob-linear})_{4}$ by $H$, and integrating by parts  with the boundary condition $\mbox{curl}H\times n|_{\partial\Omega}=0$, it shows
\begin{equation}\label{lem-4-2-4}
\eta \|\mbox{curl}H\|^2-\int_{\Omega}\mbox{curl}(u\times H)\cdot H=-\int_{\Omega}H_{t}\cdot H
\leq C\|H_{t}\|.
\end{equation}
Putting (\ref{lem-4-2-1})-(\ref{lem-4-2-4}) together, it could be arrived at
\begin{equation*}
\begin{split}
&\mu\|\mbox{curl}u\|^2+(2\mu+\lambda)\|\mbox{div}u\|^2+\kappa \|\nabla\theta\|^2+\eta \|\mbox{curl}H\|^2\\[2mm]
&\leq C\|(q_t,u_t, \theta_t,  H_{t})\|
+C(\|\theta\|_{L^{\infty}} +\|(u,q,\theta)\|_{L^{3}})\|(\nabla u, \nabla \theta, \mbox{curl}H)\|^2,
\end{split}
\end{equation*}
where we have used the fact
\begin{equation*}
 \int_{\Omega}(\mbox{curl}H\times H) \cdot u dx+\int_{\Omega}\mbox{curl}(u\times H)\cdot H dx =0.
 \end{equation*}
Taking the same argument as (\ref{lem4-1-H}) of Lemma \ref{lem4-1}, the desired results (\ref{deacy2}) is received.

In order to prove (\ref{deacy3}), we employ $(\ref{prob-linear})_{3}$ to get the estimate for $\Delta\theta$ firstly:
\begin{equation*}
\kappa\Delta\theta=
c_{v}\rho(\theta_{t}+ u\cdot\nabla\theta)+R\mbox{div}u-\lambda(\mbox{div}u)^2-2\mu |S(u)|^2+R(\rho\theta+q)\mbox{div}u-\eta|\mbox{curl}H|^2
\end{equation*}
which implies
\begin{equation}\label{4.11}
\|\Delta\theta\|^2\leq C(\|\theta_{t}\|+\|(\nabla u,\nabla\theta, \mbox{curl}H)\|^2).
\end{equation}
Recalling the stand elliptic estimate in Lemma \ref{lem-elliptic-3}:
\begin{equation}
\|\nabla^2\theta\|^2\leq C\Big(\|\Delta\theta\|^2+\|\nabla\theta\|^2\Big) \quad \mbox{for} \quad \theta\in H^{2}(\Omega)\quad\mbox{with}\quad \frac{\partial \theta}{\partial n}\Big|_{\partial\Omega}=0,
\end{equation}
which together with (\ref{deacy2}) and (\ref{4.11}) induces that
\begin{equation}
\|\nabla^2 \theta\|^2\leq C\|(q_t,u_t,\theta_t, H_t)\| \quad \forall t\geq T_2.
\end{equation}
Similarly, we have the following equation from $(\ref{prob-linear})_{4}$:
\begin{equation*}
\eta \Delta H=H_{t}-\mbox{curl}(u\times H),
\end{equation*}
which implies that
\begin{equation*}
\begin{split}
\|\Delta H\|^2&\leq \|H_{t}\|^2+\|u\|_{L^{\infty}}^2\| \nabla H\|^2+\|H\|_{L^{\infty}}\|\nabla u\|^2\\[2mm]
&\leq C\|(q_t,u_t,\theta_t, H_{t})\|.
\end{split}
\end{equation*}
Furthermore, using the standard elliptic estimate, it holds
\begin{equation*}
\|\nabla ^2H\|^2\leq C\Big(\|\Delta H\|^2+\|\nabla H\|^2\Big)
\leq C\|(q_t,u_t,\theta_t, H_{t})\|.
\end{equation*}
Therefore, the proof of Lemma \ref{lem4-2} is completed.
\end{proof}

\begin{lem}\label{lem4-3}
	For every solution $(q,u,\theta,H)\in X(0, \infty; \mathcal{E}_0)$ of the problem (\ref{prob-linear})-(\ref{boundary-linear}), it exists a time $T_3=T_3(q,u,\theta,H)\geq T_2$ such that
\begin{equation}
\|\mbox{curl}^2u\|\leq C t^{-1/4} \quad \forall t\geq T_3
\end{equation}
\end{lem}
\begin{proof}
Recalling the vorticity equation:
\begin{equation}\label{curl-equation-u-2}
\rho(w_{t}+ u\cdot\nabla w)-\mu\Delta w=K+\mbox{curl}(\mbox{curl}H\times H),
\end{equation}
where
\begin{equation*}
K=-\nabla q\times (u_t+u\cdot\nabla u)-\rho (w\cdot\nabla)u-\rho w\mbox{div}u.
\end{equation*}
Taking the $L^2$ inner product of  (\ref{curl-equation-u-2}) with $w$, using integration by parts with $w\times n|_{}\partial\Omega=0$, we obtain
\begin{equation}\label{w}
\begin{split}
\frac{1}{2}\frac{d}{dt}\|\sqrt{\rho}w\|^2+\mu\|\mbox{curl}w\|^2=\int_{\Omega}K\cdot w dx+\int_{\Omega}(\mbox{curl}H\times H)\cdot \mbox{curl}w dx.
\end{split}
\end{equation}
It is easy to check
\begin{equation*}
\int_{\Omega}K\cdot w dx\leq C\|(q_t,u_t,H_t,\theta_t)\|
\end{equation*}
and
\begin{equation*}
\begin{split}
\int_{\Omega}(\mbox{curl}H\times H)\cdot \mbox{curl} dx
&\leq \frac{\mu}{2}\|\mbox{curl}w\|^2+C\|\mbox{curl}H\times H\|^2\\[2mm]
&\leq \frac{\mu}{2}\|\mbox{curl}w\|^2+C\|(q_t,u_t,H_t,\theta_t)\|.
\end{split}
\end{equation*}
after utilizing  Sobolev's inequality and Lemma \ref{lem4-2}. Hence, (\ref{w}) becomes
\begin{equation}\label{fin-1}
\begin{split}
\frac{1}{2}\frac{d}{dt}\|\sqrt{\rho}w\|^2+\frac{\mu}{2}\|\mbox{curl}w\|^2\leq C\|(q_t,u_t,H_t,\theta_t)\|.
\end{split}
\end{equation}
On the other hand,  multiplication (\ref{curl-equation-u-2}) by $w_t$ followed by integration over $\Omega$ gives that
\begin{equation}\label{w_t}
\begin{split}
\frac{\mu}{2}\frac{d}{dt}\|\mbox{curl}w\|^2+\|\sqrt{\rho}w_t\|^2=-\int_{\Omega}\rho u\cdot\nabla w \cdot w_{t}dx+\int_{\Omega}K\cdot w_{t}dx+\int_{\Omega}\mbox{curl}(\mbox{curl}H\times H)\cdot w_{t} dx
\end{split}
\end{equation}
Applying   Sobolev's inequality  and Lemma \ref{lem4-2},  we obtain
\begin{equation}
\begin{split}
\int_{\Omega}\rho u\cdot\nabla w \cdot w_{t} dx&\leq C\|u\|_{L^{\infty}}\|\nabla w\| \|\sqrt{\rho}w_{t}\| \\[2mm]
&\leq  \frac{1}{8}\|\sqrt{\rho}w_{t}\| ^2+C\|\nabla u\|^2_{1}\|\nabla w\|^2\\[2mm]
&\leq \frac{1}{8}\|\sqrt{\rho}w_{t}\| ^2+C\|\nabla u\|^2_{1}\|\mbox{curl} w\|^2+C\|\nabla u\|^2_{1}\|w\|^2\\[2mm]
&\leq \frac{1}{8}\|\sqrt{\rho}w_{t}\| ^2+C\|\nabla u\|^2_{1}\|\mbox{curl} w\|^2+C\|(q_t,u_t,H_t,\theta_t)\|,
\end{split}
\end{equation}
\begin{equation}
\begin{split}
\int_{\Omega}K\cdot w_{t} dx&\leq  \frac{1}{8}\|\sqrt{\rho}w_{t}\| ^2 +C\|(q_t,u_t,H_t,\theta_t)\|,
\end{split}
\end{equation}
and
\begin{equation}
\begin{split}
\int_{\Omega}\mbox{curl}(\mbox{curl}H\times H)\cdot w_{t} dx&\leq \frac{1}{8} \|\sqrt{\rho}w_{t}\|^2 +C \|\nabla H\|^2_{1}\|\nabla ^2H\|^2\\[2mm]
&\leq \frac{1}{8} \|\sqrt{\rho}w_{t}\|^2 +C \|(q_t,u_t,H_t,\theta_t)\|,
\end{split}
\end{equation}
then substituting these into (\ref{w_t}), we derive that
\begin{equation}\label{fin-2}
\begin{split}
&\frac{\mu}{2}\frac{d}{dt}\|\mbox{curl}w\|^2+\frac{5}{8}\|\sqrt{\rho}w_t\|^2\\[2mm]
&\leq C\|\nabla u\|^2_{1}\|\mbox{curl} w\|^2+C\|(q_t,u_t,H_t,\theta_t)\|.
\end{split}
\end{equation}
Summing (\ref{fin-1}) and (\ref{fin-2}) up,  one gets
 \begin{equation}
 \begin{split}
		&\frac{d}{dt}\left\{\|\sqrt{\rho}w\|^2+\mu\|\mbox{curl}w\|^2\right\}
+c_{0} \left\{\|\mbox{curl}w\|^2+\|\sqrt{\rho}w_t\|^2\right\}\\
&\leq C\|\nabla u\|^2_{1} \|\mbox{curl}w\|^2
+C\|(q_t,u_t,H_t,\theta_t)\|,
 \end{split}
	\end{equation}
for some positive constant $c_0>0$. From (\ref{tinfty}), there exists $T'_3$  such that if $t\geq T'_3$, then
\begin{equation*}
C\|\nabla u\|^2_{1}\leq\frac{c_0}{2},
\end{equation*}
which yields
 \begin{equation*}
 \frac{d}{dt}\left\{\|\sqrt{\rho}w\|^2+\mu\|\mbox{curl}w\|^2\right\}
+\frac{c_{0}}{2} \left\{\|\mbox{curl}w\|^2+\|\sqrt{\rho}w_t\|^2\right\}
\leq C\|(q_t,u_t,H_t,\theta_t)\|.
	\end{equation*}
So,
 \begin{equation*}
 \frac{d}{dt}\left\{\|\sqrt{\rho}w\|^2+\mu\|\mbox{curl}w\|^2\right\}
+\frac{c_{0}}{2} \|\mbox{curl}w\|^2
\leq C\|(q_t,u_t,H_t,\theta_t)\|.
	\end{equation*}
Putting (\ref{deacy2}) together with the above inequality induce that
\begin{equation*}
 \frac{d}{dt}\left\{\|\sqrt{\rho}w\|^2+\mu\|\mbox{curl}w\|^2\right\}
+\frac{c_{0}}{2} \left\{\|\mbox{curl}w\|^2+\|\sqrt{\rho}w\|^2\right\}
\leq C\|(q_t,u_t,H_t,\theta_t)\|, \quad \forall t>T_3=\max\{T'_3, T_2\}
\end{equation*}
then  from Lemma \ref{lem4-1}, we have
\begin{equation*}
g'(t)+\frac{c_0}{2}g(t)\leq Ct^{-1/2},\,\,\, \forall t>T_3
\end{equation*}
where $g(t)$ is defined as
\begin{equation*}
g(t)\equiv\|\sqrt{\rho}w\|^2+\mu\|\mbox{curl}w\|^2.
\end{equation*}
An application of Lemma \ref{lem-deacy} (ii) implies
\begin{equation*}
g(t)\leq Ct^{-1/2} \quad \mbox{for} \quad t\geq T_3.
\end{equation*}
Therefore, this completes the proof of Lemma \ref{lem4-3}.
\end{proof}

\begin{lem}\label{lem4-4}
For every solution $(q,u,\theta,H)\in X(0, \infty; \mathcal{E}_0)$ of the problem (\ref{prob-linear})-(\ref{boundary-linear}), it exists a time $T_4\geq T_3$, such that
\begin{equation}\label{deacy4}
\|\nabla^2 u\|+\|\nabla q\|\leq C t^{-1/4} \quad \forall t\geq T_4.
\end{equation}
\end{lem}

\begin{proof} Consider $(q, u)$ as the solution of the following Stokes problem:
\begin{equation}\label{Stoke}
\left\{
\begin{array}{llll}
	-\mu \Delta u+ R\nabla q=G \quad \mbox{in}\quad \Omega\\[2mm]
	\mbox{div}u=\mbox{div}u \quad \mbox{in}\quad \Omega\\[2mm]
	u\cdot n|_{\partial\Omega}=0,\quad \mbox{curl}u\times n|_{\partial\Omega}=0
\end{array}
\right.
\end{equation}
where
\begin{equation*}
	\begin{split}
G&=(\mu+\lambda)\nabla\mbox{div}u-\rho\big(u_{t}+u\cdot\nabla u\big)-R\nabla\theta-R\nabla(q\theta)+H\cdot\nabla H-\frac{1}{2}\nabla|H|^2\\[2mm]
&\equiv(\mu+\lambda)\nabla\mbox{div}u+G_{1}
\end{split}
\end{equation*}
In view of H$\ddot{o}$lder's inequality, Sobolev's inequality and Lemma \ref{lem4-2}, it is easy to see that
\begin{equation}\label{G_1}
\|G_1\|^2\leq C\|(q_t,u_t,\theta_t,H_t)\|+\|\theta\|_{L^{\infty}}^2\|\nabla q\|^2,
\end{equation}
which implies that
\begin{equation*}
\|G\|^2
\leq C\|\nabla\mbox{div}u\|^2+C\|(q_t,u_t,\theta_t,H_t)\|+\|\theta\|_{L^{\infty}}^2\|\nabla q\|^2.
\end{equation*}
Then elliptic estimate (see Lemma \ref{lem-elliptic-2}) for the Stoke problem yields
\begin{equation*}
\begin{split}
\|\nabla^2u\|^2+\|\nabla q\|^2&\leq C\left(\|G\|^2+\|\mbox{div}u\|^2_{1}+\|\nabla u\|^2\right)\\[2mm]
&\leq C\|\nabla\mbox{div}u\|^2+C\|(q_t,u_t,\theta_t,H_t)\|+\|\theta\|_{L^{\infty}}^2\|\nabla q\|^2.
\end{split}
\end{equation*}
By (\ref{tinfty}) and Lemma \ref{sobolev inequ}, there exists a time $T_{4}>T_{3}$ such that
\begin{equation}\label{small}
\|\theta\|_{L^{\infty}}^2 \leq C\|\nabla \theta\|^2_1\leq \frac{1}{4}.
\end{equation}
Therefore, we obtain the following inequality
\begin{equation}\label{u2}
\|\nabla^2u\|^2+\frac{3}{4}\|\nabla q\|^2
\leq C\|\nabla\mbox{div}u\|^2+C\|(q_t,u_t,\theta_t,H_t)\|
\end{equation}
 for $t\geq T_4$. Next, we shall estimate the term $\|\nabla\mbox{div}u\|$. For this purpose, we take the derivative of the continuity equation with respect to $x_{i}$, $1\leq i\leq 3$ as follows
\begin{equation}\label{divu}
\begin{split}
(\mbox{div}u)_{x_{i}}&=-\frac{1}{\rho}\left( q_{t,x_i}+ q_{x_i}\mbox{div}u+u\cdot\nabla q_{x_{i}}+u_{x_{i}}\cdot\nabla q\right)\\[2mm]
&\equiv-\frac{1}{\rho} q_{t,x_i}-\frac{1}{\rho}F^{i}.
\end{split}
\end{equation}
It is clear that
\begin{equation}\label{F}
\|F\|^2\leq C\|\nabla q\|^2_{2}\|\nabla u\|^2\leq C\|(q_t,u_t,H_t,\theta_t)\|,
\end{equation}
which together with (\ref{divu}) implies that
\begin{equation}\label{divu2}
\|\nabla\mbox{div}u\|^2\leq C\int_{\Omega}|\nabla q_{t}|^2 dx +C\|(q_t,u_t,H_t,\theta_t)\|.
\end{equation}
In order to deal with $\|\nabla q_{t}\|^2$, we follow the method given in \cite{1992MZ}.  Let $d(x):=dist(x,\partial\Omega)$ satisfies
 $$d\in C^2(\bar{\Gamma}_h), \Gamma_h=\{x\in\Omega|0\leq d(x)<h\}$$
and $\psi\in C^{\infty}(\mathbb{R}^{+})$ be non-increasing
with $\psi(d)=1$ for $0\leq d\leq \frac{1}{2}h$, and $\psi(d)=0$ for $d\geq\frac{3}{4}h$. For $1\leq i,k\leq 3$, we define
\begin{equation*}
e_{ik}(x):=\delta_{ik}-\psi(d(x))^2d_{x_i}(x)d_{x_{k}}(x).
\end{equation*}
For any smooth function $f:\Omega\rightarrow \mathbb{R}$, one has
\begin{equation}\label{f-1}
\begin{split}
&f_{x_i}=e_{ik}f_{x_{k}}+\psi(d)^2d_{x_i}d_{x_k}f_{x_{k}}\quad 1\leq i\leq 3,\\[1mm]
&|\nabla f|^2=e_{ik}f_{x_i}f_{x_k}+\psi(d)^2(\nabla f\cdot\nabla d)^2.
\end{split}
\end{equation}
Since $\psi(0)=1$ and $\nabla d(x)=n(x)$ for any $x\in\partial\Omega$, where $n$ denotes the outer unit normal, we have
\begin{equation}\label{f-2}
\begin{split}
&e_{ik}n_{k}|_{\partial\Omega}=0,\quad\quad 1\leq i\leq3,\\[1mm]
&e_{ik}f_{x_k}n_{i}|_{\partial\Omega}=0.
\end{split}
\end{equation}
Therefore, we have the following equivalent form
\begin{equation}\label{f-3}
\int_{\Omega}|\nabla q_{t}|^2 dx=\int_{\Omega}e_{ik} q_{t,x_i}q_{t,x_{k}} dx+\int_{\Omega}\psi(d)^2(\nabla q_{t}\cdot \nabla d)^2 dx,
\end{equation}
The first term on the right-hand side of (\ref{f-3}) could be controlled as
\begin{equation}\label{qt}
\begin{split}
\int_{\Omega}e_{ik} q_{t,x_i}q_{t,x_{k}} dx&=-\int_{\Omega}(e_{ik})_{x_i}q_{t}q_{t,x_k} dx
-\int_{\Omega}e_{ik}q_{t}q_{t,x_{i}x_{k}} dx+\int_{\partial\Omega}e_{ik}n_{i}q_{t}q_{t,x_{k}} dx\\[2mm]
&=-\int_{\Omega}(e_{ik})_{x_i}q_{t}q_{t,x_k} dx
-\int_{\Omega}e_{ik}q_{t}q_{t,x_{i}x_{k}} dx\\[2mm]
&\leq C\|q_t\|\|\nabla q_{t}\|_{1}\leq C\|q_t\|.
\end{split}
\end{equation}
by integrating  by parts with the  boundary condition (\ref{f-2}).
Therefore, combining (\ref{u2}), (\ref{divu2}) with (\ref{f-3})-(\ref{qt})  gives that
\begin{equation}\label{4-0}
\|\nabla^2u\|^2+\|\nabla q\|^2
\leq C\|\psi(d)(\nabla q_{t}\cdot \nabla d)\|^2+C\|(q_t,u_t,\theta_t,H_t)\|,\quad \forall t\geq T_{3}.
\end{equation}
Multiply (\ref{divu}) by $(2\mu+\lambda)\psi(d)d_{x_i}$, and sum up all these equations for $i=1,2,3$, then we get
\begin{equation}\label{4-1}
\frac{2\mu+\lambda}{\rho}\psi(d)\nabla q_{t}\cdot\nabla d+(2\mu+\lambda)\psi(d)\nabla\mbox{div}u\cdot \nabla d
=\frac{2\mu+\lambda}{\rho}\psi(d)F\cdot \nabla d
\end{equation}
In order to cancel the second term on the left-hand side on (\ref{4-1}), we rewrite $(\ref{Stoke})_{1}$ as
\begin{equation*}
-(2\mu+\lambda)\nabla\mbox{div}u+R\nabla q=-\mu \mbox{curl curl}u+G_1
\end{equation*}
then multiplying this identity by $\psi(d)\nabla d$, and adding the result to (\ref{4-1}), we infer that
\begin{equation}\label{Q-equation}
\begin{split}
&\frac{2\mu+\lambda}{\rho}\psi(d)\nabla q_{t}\cdot\nabla d+R\psi(d)\nabla q\cdot\nabla d\\[2mm]
&=\frac{2\mu+\lambda}{\rho}\psi(d)F\cdot \nabla d-\mu\psi(d) \mbox{curl curl}u\cdot\nabla d+\psi(d)G_1 \cdot\nabla d\equiv Q.
\end{split}
\end{equation}
Moreover, (\ref{F}), (\ref{G_1}), Lemma \ref{lem4-1} and Lemma \ref{lem4-3} imply that
\begin{equation}\label{fin-6}
\begin{split}
\|Q\|^2&\leq C\|\mbox{curl}\mbox{curl}u\|^2+\|F\|^2+\|G_1\|^2\\[2mm]
&\leq C\|\theta\|_{L^{\infty}}^2\|\nabla q\|^2+Ct^{-1/2},
\end{split}
\end{equation}
where we have used the definition of $\psi(d)$ and $|\nabla n|=1$. Therefore, (\ref{small}) and (\ref{4-0}) induce
\begin{equation}\label{fin-5}
\|\nabla^2u\|^2+\frac{1}{2}\|\nabla q\|^2
\leq C\|\psi(d)\nabla q\cdot\nabla d\|^2+Ct^{-1/2}, \quad
\mbox{for}\quad t\geq \widetilde{T_{4}}.
\end{equation}

On the other side,  multiplying  (\ref{Q-equation}) by $\psi(d)\nabla q\cdot\nabla d$, it indicates
\begin{equation*}
\begin{split}
&\frac{d}{dt}\int_{\Omega}\frac{2\mu+\lambda}{\rho}(\psi(d)\nabla q\cdot\nabla d)^2 dx+R\|\psi(d)\nabla q\cdot \nabla d\|^2\\[2mm]
&=\int_{\Omega}(\psi(d)\nabla q\cdot\nabla d)Q dx+\int_{\Omega}(\frac{2\mu+\lambda}{\rho})_{t}\psi(d)\nabla q\cdot\nabla d dx\\[2mm]
&\leq \varepsilon \|\psi(d)\nabla q\cdot \nabla d\|^2+C_\varepsilon (\|Q\|^2+\|q_t\|^2).
\end{split}
\end{equation*}
Therefore, from Lemma \ref{lem4-1}, (\ref{tinfty}), (\ref{fin-6}) and (\ref{fin-5}),  we conclude that
\begin{equation*}
\begin{split}
\frac{d}{dt}\int_{\Omega}\frac{2\mu+\lambda}{\rho}(\psi(d)\nabla q\cdot\nabla d)^2 dx+\frac{R}{2}\|\psi(d)\nabla q\cdot \nabla d\|^2
\leq  Ct^{-1/2}, \quad \mbox{for}\quad t\geq T_{4}\geq \widetilde{T_4}.
\end{split}
\end{equation*}
 One can apply (ii) of Lemma \ref{lem-deacy} to gain
\begin{equation*}
\int_{\Omega}\frac{2\mu+\lambda}{\rho}(\psi(d)\nabla q\cdot\nabla d)^2 dx\leq Ct^{-1/2},
\end{equation*}
which together with (\ref{fin-5}) finishes the proof of Lemma \ref{lem4-4}.
\end{proof}

Now, we are ready to prove  Theorem \ref{thm-2}.

\noindent{\bf {Proof of  Theorem \ref{thm-2} }}
Lemma \ref{lem4-2} and Lemma \ref{lem4-4} yield
\begin{equation*}
\|(u,\theta,H)\|_{C^{0}(\bar{\Omega})}\leq C\|(\nabla u, \nabla \theta, \mbox{curl}H)\|_{1}\leq  Ct^{-1/4}.
\end{equation*}
Because
\begin{equation*}
\|q\|^4_{L^{4}}\leq \|q\| \|q\|^3_{L^6}\leq C\|\nabla q\|^3\leq Ct^{-3/4},
\end{equation*}
and
\begin{equation*}
\|\nabla q\|^4_{L^{4}}\leq \|\nabla q\|^2_{L^{\infty}} \|\nabla q\|^2\leq Ct^{-1/2}.
\end{equation*}
Then we have
\begin{equation*}
\|\rho-1\|_{C^{0}(\bar{\Omega})}=\|q\|_{C^{0}(\bar{\Omega})}\leq C\|q\|_{H^{1,4}}\leq C\left(t^{-3/16}+t^{-1/8}\right)\leq C t^{-1/8}.
\end{equation*}
 Hence Theorem \ref{thm-2} is proved.
\qed

\centerline{\bf Acknowledgements}
Liu's research is supported by National Natural Science Foundation of China (No.12071219, 12026432). Luo's research is supported by a grant from the Research Grants Council of the Hong Kong Special Administrative Region, China (Project No. 11307420). Zhong is supported by the Fundamental Research Funds for the Central Universities (No. A0920502052101-224). Luo would like to thank Professor Chongchun Zeng for helpful discussions on the Hodge type estimates.

\end{document}